\pdfoutput=1
\documentclass[reqno,a4paper,11pt]{amsart}
\usepackage[dvipsnames,svgnames,table]{xcolor}
\usepackage[utf8]{inputenc}
\usepackage[final]{microtype}
\usepackage{amsmath,amssymb,amsthm}
\usepackage{mathrsfs}
\usepackage{mathtools}
\usepackage{commath}
\usepackage{thmtools}
\usepackage{thm-restate}
\usepackage{cases}
\usepackage{enumitem}
\setlist[enumerate,1]{label=(\roman*)}
\usepackage[pdfborderstyle={/S/U/W 0},unicode]{hyperref}
\hypersetup{
    colorlinks=true,
    linkcolor=blue!85!black,
    citecolor=blue!85!black,
    urlcolor=blue!85!black,
}
\usepackage{cleveref}
\usepackage{etoolbox}
\numberwithin{equation}{section}
\usepackage{thm-restate}
\usepackage{tikz}
\usepackage{tkz-euclide}
\usetikzlibrary{arrows.meta}
\crefname{subsection}{subsection}{subsections}
\usepackage{tikz-cd}
\tikzset{smalltext/.style={"\textup{\small #1}" description}}
\usepackage{pgfplots}
\pgfplotsset{compat = newest}
\usetikzlibrary{decorations.pathreplacing}

\theoremstyle{plain}
\newtheorem{theorem}{Theorem}

\newtheorem{corollary}[theorem]{Corollary}

\theoremstyle{definition}

\DeclareMathOperator{\Li}{Li}
\DeclareMathOperator{\aff}{aff}

\DeclareMathOperator{\Graff}{Graff}

\newcommand{\join}{\vee}
\newcommand{\meet}{\wedge}

\newcommand{\defined}{\mathrel{\coloneqq}}

\newcommand{\divides}{\mathrel{\mid}}

\DeclareMathOperator{\ind}{\mathbf{1}}

\newcommand{\eps}{\varepsilon}
\let\SS\relax
\newcommand{\CC}{\mathbb{C}}
\newcommand{\RR}{\mathbb{R}}
\newcommand{\SS}{\mathbb{S}}

\newcommand{\cC}{\mathcal{C}}
\newcommand{\cD}{\mathcal{D}}
\newcommand{\cG}{\mathcal{G}}
\newcommand{\cI}{\mathcal{I}}
\newcommand{\cJ}{\mathcal{J}}

\newcommand{\cX}{\mathcal{X}}

\DeclareMathOperator{\Unif}{Unif}

\renewcommand{\d}{\mathop{}\!\mathrm{d}}
\newcommand{\dt}{\d t}
\newcommand{\dx}{\d x}
\newcommand{\dy}{\d y}
\newcommand{\dr}{\d r}
\newcommand{\du}{\d u}
\newcommand{\dtheta}{\d \theta}

\undef{\set}
\DeclarePairedDelimiter{\set}{\lbrace}{\rbrace}
\DeclarePairedDelimiter{\p}{\lparen}{\rparen}
\newcommand{\from}{\colon}
\newcommand{\st}{\mathbin{\colon}}
\undef{\abs}
\DeclarePairedDelimiterX{\abs}[1]
  {\lvert}{\rvert}{\ifblank{#1}{\,\cdot\,}{#1}}

\DeclareMathDelimiter{\given}
  {\mathbin}{symbols}{"6A}{largesymbols}{"0C}
\DeclareMathOperator{\Prob}{\mathbb{P}}
\DeclarePairedDelimiterXPP{\prob}[1]
  {\Prob}{\lparen}{\rparen}{}
  {\renewcommand{\given}{\nonscript\;\delimsize\vert\nonscript\;\mathopen{}}#1}
\DeclareMathOperator{\Expec}{\mathbb{E}}
\DeclarePairedDelimiterXPP{\expec}[1]
  {\Expec}{\lparen}{\rparen}{}
  {\renewcommand{\given}{\nonscript\;\delimsize\vert\nonscript\;\mathopen{}}#1}

\usepackage[a4paper,margin=1.2in]{geometry}

\usepackage[backend=biber, style=ext-numeric, articlein=false, maxnames=99, giveninits=true, doi=false, isbn=false, url=false]{biblatex}
\bibliography{karamata-intersection-arxiv-v1}

\begin{document}

\title{Intersections of random chords of a circle}

\author{Cynthia Bortolotto}
\address{Department of Mathematics, ETH Zürich, Zürich, Switzerland}
\email{cynthia.bortolotto@math.ethz.ch}

\author{Victor Souza}
\address{Department of Computer Science and Technology, and Sidney Sussex College, University of Cambridge, Cambridge, United Kingdom}
\email{vss28@cam.ac.uk}

\begin{abstract}
Where are the intersection points of diagonals of a regular $n$-gon located?
What is the distribution of the intersection point of two random chords of a circle?
We investigate these and related new questions in geometric probability, extend a largely forgotten result of Karamata, and elucidate its connection to the Bertrand paradox.
\end{abstract}

\maketitle


\section{Introduction}

Draw all the diagonals of a regular $n$-gon and consider their intersection points.
The resulting picture is as in \Cref{fig:diagonals}.
Drawings of this kind are very popular, possibly as they offer an easy method for creating symmetrical and intricate designs.

\begin{figure}[!ht]
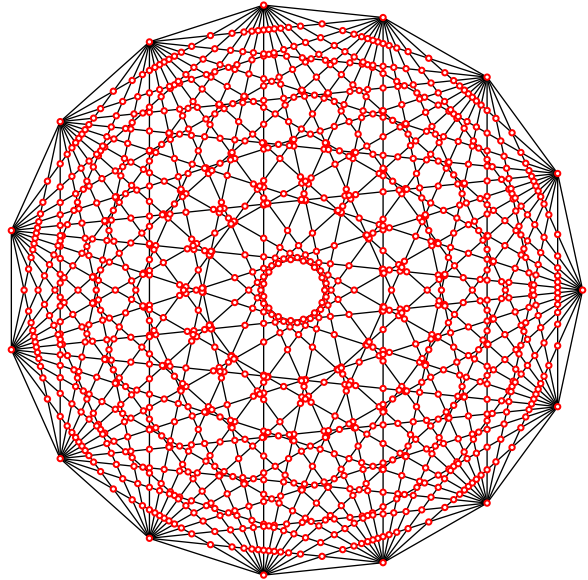

\centering

\caption{A regular 15-gon  with all diagonals and their intersection points.}
\label{fig:diagonals}
\end{figure}

Beyond its artistic applications, this drawing has been the subject of serious mathematical investigations.
Indeed, the following geometric problem was first studied by the Serbian mathematician Jovan Karamata~\cite{Karamata1962-ec} in 1962.
Which proportion of the intersection points of pairs of diagonals of a regular $n$-gon lie inside a disk $\cD_r$ of radius $r$, centered at the origin?
Karamata gave a surprising answer to this question.
Write $I_n(r)$ of the number of intersection points in inside $\cD_r$, counted with multiplicity.
Then 
\begin{align}
\label{eq:karamata-limit}
    \lim_{n \to \infty} \frac{I_n(r)}{I_n(1)} = \frac{6}{\pi^2} \Li_2(r^2), \quad \text{for } 0 \leq r \leq 1,
\end{align}
where $\Li_2$ is Euler's dilogarithm function, defined by the convergent series
\begin{equation*}
    \Li_2(x) \defined \sum_{k=0}^\infty \frac{x^k}{k^2}, \quad \text{for } 0 \leq x \leq 1.
\end{equation*}

One of the remarkable features of this answer is the surprising appearance of a special function from a rather innocent question in planar geometry.
The function $\Li_2$, also known as Spence's function, can be extended analytically to $z \in \CC \setminus [1,\infty)$ as
\begin{equation*}
    \Li_2(z) = - \int_{0}^{z} \frac{\log (1 - u)}{u} \du.
\end{equation*}
The fact that this function cannot be extended analytically in the real line beyond the point $x = 1$ can be explained in light of the original problem.
In \Cref{fig:dilog}, we plot the limit \eqref{eq:karamata-limit} as a function of $r$, together with its derivative, which indicates the local density of points that lie on a specific distance $0 \leq r \leq 1$ from the origin.
Note that the density indeed blows up at $r = 1$.

\begin{figure}[!ht]
\centering
\begin{tikzpicture}
\definecolor{col_a}{RGB}{230,97,1};
\definecolor{col_b}{RGB}{253,184,99};
\definecolor{col_c}{RGB}{178,171,210};
\definecolor{col_d}{RGB}{94,60,153};
\begin{axis}[
    set layers,
    axis line style={on layer=axis foreground, line width = 1pt},
    width=220pt,
    height=140pt,
    axis lines=left,
    xmin=0.0, xmax=1.05,
    ymin=0, ymax=3,
    xtick={0, 1},
    ytick={1.0},
    xticklabels={$0$,$1$},
    yticklabels={$1$},
    legend pos=north west,
    ymajorgrids=true,
    xmajorgrids=true,
    grid style=dashed,
]
\addplot[color = col_a, line width=1pt]
coordinates {
    (0., 0.)(0.005, 0.0000151983)(0.01, 0.0000607942)(0.015, 0.000136791)(0.02, 0.000243195)(0.025, 0.000380014)(0.03, 0.000547258)(0.035, 0.000744939)(0.04, 0.000973073)(0.045, 0.00123168)(0.05, 0.00152077)(0.055, 0.00184037)(0.06, 0.00219051)(0.065, 0.00257121)(0.07, 0.0029825)(0.075, 0.00342441)(0.08, 0.00389698)(0.085, 0.00440023)(0.09, 0.00493422)(0.095, 0.00549897)(0.1, 0.00609454)(0.105, 0.00672096)(0.11, 0.00737829)(0.115, 0.00806658)(0.12, 0.00878587)(0.125, 0.00953623)(0.13, 0.0103177)(0.135, 0.0111304)(0.14, 0.0119743)(0.145, 0.0128495)(0.15, 0.0137561)(0.155, 0.0146941)(0.16, 0.0156637)(0.165, 0.0166648)(0.17, 0.0176977)(0.175, 0.0187623)(0.18, 0.0198587)(0.185, 0.0209871)(0.19, 0.0221475)(0.195, 0.02334)(0.2, 0.0245647)(0.205, 0.0258217)(0.21, 0.0271111)(0.215, 0.028433)(0.22, 0.0297876)(0.225, 0.0311748)(0.23, 0.032595)(0.235, 0.034048)(0.24, 0.0355342)(0.245, 0.0370535)(0.25, 0.0386062)(0.255, 0.0401924)(0.26, 0.0418121)(0.265, 0.0434655)(0.27, 0.0451529)(0.275, 0.0468742)(0.28, 0.0486297)(0.285, 0.0504195)(0.29, 0.0522438)(0.295, 0.0541027)(0.3, 0.0559964)(0.305, 0.057925)(0.31, 0.0598888)(0.315, 0.0618878)(0.32, 0.0639224)(0.325, 0.0659926)(0.33, 0.0680986)(0.335, 0.0702407)(0.34, 0.072419)(0.345, 0.0746338)(0.35, 0.0768852)(0.355, 0.0791735)(0.36, 0.0814988)(0.365, 0.0838614)(0.37, 0.0862616)(0.375, 0.0886994)(0.38, 0.0911753)(0.385, 0.0936894)(0.39, 0.096242)(0.395, 0.0988332)(0.4, 0.101464)(0.405, 0.104133)(0.41, 0.106842)(0.415, 0.109591)(0.42, 0.11238)(0.425, 0.115209)(0.43, 0.118079)(0.435, 0.12099)(0.44, 0.123942)(0.445, 0.126936)(0.45, 0.129972)(0.455, 0.13305)(0.46, 0.136171)(0.465, 0.139334)(0.47, 0.142541)(0.475, 0.145792)(0.48, 0.149086)(0.485, 0.152425)(0.49, 0.155809)(0.495, 0.159238)(0.5, 0.162713)(0.505, 0.166234)(0.51, 0.169802)(0.515, 0.173416)(0.52, 0.177078)(0.525, 0.180787)(0.53, 0.184546)(0.535, 0.188352)(0.54, 0.192209)(0.545, 0.196115)(0.55, 0.200071)(0.555, 0.204079)(0.56, 0.208138)(0.565, 0.212249)(0.57, 0.216413)(0.575, 0.22063)(0.58, 0.224901)(0.585, 0.229227)(0.59, 0.233608)(0.595, 0.238044)(0.6, 0.242538)(0.605, 0.247088)(0.61, 0.251697)(0.615, 0.256364)(0.62, 0.261091)(0.625, 0.265879)(0.63, 0.270727)(0.635, 0.275638)(0.64, 0.280612)(0.645, 0.285649)(0.65, 0.290752)(0.655, 0.29592)(0.66, 0.301155)(0.665, 0.306458)(0.67, 0.311829)(0.675, 0.317271)(0.68, 0.322783)(0.685, 0.328368)(0.69, 0.334027)(0.695, 0.33976)(0.7, 0.345569)(0.705, 0.351456)(0.71, 0.357421)(0.715, 0.363467)(0.72, 0.369595)(0.725, 0.375806)(0.73, 0.382102)(0.735, 0.388485)(0.74, 0.394956)(0.745, 0.401518)(0.75, 0.408172)(0.755, 0.41492)(0.76, 0.421765)(0.765, 0.428708)(0.77, 0.435752)(0.775, 0.4429)(0.78, 0.450153)(0.785, 0.457515)(0.79, 0.464987)(0.795, 0.472574)(0.8, 0.480278)(0.805, 0.488103)(0.81, 0.496051)(0.815, 0.504126)(0.82, 0.512333)(0.825, 0.520674)(0.83, 0.529156)(0.835, 0.537781)(0.84, 0.546554)(0.845, 0.555482)(0.85, 0.564569)(0.855, 0.573821)(0.86, 0.583244)(0.865, 0.592846)(0.87, 0.602633)(0.875, 0.612614)(0.88, 0.622796)(0.885, 0.63319)(0.89, 0.643805)(0.895, 0.654652)(0.9, 0.665743)(0.905, 0.677092)(0.91, 0.688714)(0.915, 0.700624)(0.92, 0.712843)(0.925, 0.72539)(0.93, 0.73829)(0.935, 0.75157)(0.94, 0.765262)(0.945, 0.779403)(0.95, 0.794036)(0.955, 0.809216)(0.96, 0.825007)(0.965, 0.841491)(0.97, 0.858775)(0.975, 0.877001)(0.98, 0.89637)(0.985, 0.917188)(0.99, 0.939976)(0.995, 0.965839)(1., 1.)
};
\addlegendentry{$\operatorname{6 \Li_2(r^2) / \pi^2}$}
\addplot[
    domain=0:0.99,
    samples=100,
    color=col_c,
    line width=1pt]
{-12*ln(1-x^2)/(pi^2*x)};
\addlegendentry{$-12 \log(1 - r^2)/\pi^2 r$}
\end{axis}
\end{tikzpicture}
\caption{Multiple of dilogarithm function and it's derivative.}
\label{fig:dilog}
\end{figure}
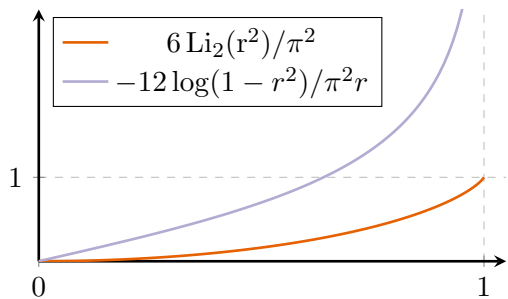

Karamata's theorem was published in a book of essays in honor of George Pólya on the occasion of his 75th birthday.
Unfortunately, this gem seems to have been overlooked ever since.
In this article, we shine some new light on his result.

One specific feature of the approach of Karamata is that intersection points are counted with multiplicity.
See \Cref{fig:diagonals-multiplicity} for an example of how arithmetical coincidences can allow for points to be hit by multiple diagonals.
\begin{figure}[!ht]
\centering
\begin{tikzpicture}
\pgfmathsetmacro{\sscale}{0.72}
\pgfmathsetmacro{\sizea}{1.5}
\pgfmathsetmacro{\sizeb}{2}
\pgfmathsetmacro{\sizec}{2.5}
\pgfmathsetmacro{\sized}{3}
\pgfmathsetmacro{\sizee}{3.2}
\definecolor{col_a}{RGB}{230,97,1};
\definecolor{col_b}{RGB}{253,184,99};
\definecolor{col_c}{RGB}{178,171,210};
\definecolor{col_d}{RGB}{94,60,153};
\tkzSetUpStyle[color=black,line width=0.5pt]{lin}
\tkzDefPoints{0/0/O, \sscale*4/0/X1}
\tkzDefPointsBy[rotation=center O angle 360/12](X1){X2}
\tkzDefPointsBy[rotation=center O angle 360/12](X2){X3}
\tkzDefPointsBy[rotation=center O angle 360/12](X3){X4}
\tkzDefPointsBy[rotation=center O angle 360/12](X4){X5}
\tkzDefPointsBy[rotation=center O angle 360/12](X5){X6}
\tkzDefPointsBy[rotation=center O angle 360/12](X6){X7}
\tkzDefPointsBy[rotation=center O angle 360/12](X7){X8}
\tkzDefPointsBy[rotation=center O angle 360/12](X8){X9}
\tkzDefPointsBy[rotation=center O angle 360/12](X9){X10}
\tkzDefPointsBy[rotation=center O angle 360/12](X10){X11}
\tkzDefPointsBy[rotation=center O angle 360/12](X11){X12}
\tkzDrawSegments[lin](X1,X2 X1,X3 X1,X4 X1,X5 X1,X6 X1,X7 X1,X8 X1,X9 X1,X10 X1,X11 X1,X12)
\tkzDrawSegments[lin](X2,X3 X2,X4 X2,X5 X2,X6 X2,X7 X2,X8 X2,X9 X2,X10 X2,X11 X2,X12)
\tkzDrawSegments[lin](X3,X4 X3,X5 X3,X6 X3,X7 X3,X8 X3,X9 X3,X10 X3,X11 X3,X12)
\tkzDrawSegments[lin](X4,X5 X4,X6 X4,X7 X4,X8 X4,X9 X4,X10 X4,X11 X4,X12)
\tkzDrawSegments[lin](X5,X6 X5,X7 X5,X8 X5,X9 X5,X10 X5,X11 X5,X12)
\tkzDrawSegments[lin](X6,X7 X6,X8 X6,X9 X6,X10 X6,X11 X6,X12)
\tkzDrawSegments[lin](X7,X8 X7,X9 X7,X10 X7,X11 X7,X12)
\tkzDrawSegments[lin](X8,X9 X8,X10 X8,X11 X8,X12)
\tkzDrawSegments[lin](X9,X10 X9,X11 X9,X12)
\tkzDrawSegments[lin](X10,X11 X10,X12)
\tkzDrawSegments[lin](X11,X12)
\foreach \x / \y in {
0.845/0.845,0.309/1.155,1.155/0.309,1.000/1.732,1.732/1.000,0.000/2.000,2.000/0.000,2.392/0.928,2.536/0.392,0.392/2.536,1.608/2.000,0.928/2.392,2.000/1.608,2.619/1.155,2.845/0.309,1.691/2.309,0.309/2.845,3.000/0.268,2.732/1.268,2.464/1.732,1.732/2.464,0.268/3.000,1.268/2.732,3.072/0.536,2.928/1.072,1.072/2.928,2.000/2.392,0.536/3.072,2.392/2.000,2.309/2.309,3.155/0.845,0.845/3.155,3.000/1.732,3.464/0.000,1.732/3.000,0.000/3.464,0.144/3.464,1.856/2.928,3.464/0.144,3.072/1.608,1.608/3.072,2.928/1.856,2.000/2.845,2.845/2.000,0.309/3.464,3.155/1.464,1.464/3.155,3.464/0.309,0.536/3.464,1.268/3.268,3.464/0.536,3.268/1.268,2.196/2.732,2.732/2.196,0.928/3.464,2.536/2.536,3.464/0.928,2.309/1.690,1.154/2.618
}{
    \filldraw[xshift=\sscale*\x cm, yshift=\sscale*\y cm, black, line width=0.8pt, fill=Cerulean] (0,0) circle (\sizea pt);
    \filldraw[xshift=\sscale*\x cm, yshift=-\sscale*\y cm, black, line width=0.8pt, fill=Cerulean] (0,0) circle (\sizea pt);
    \filldraw[xshift=-\sscale*\x cm, yshift=\sscale*\y cm, black, line width=0.8pt, fill=Cerulean] (0,0) circle (\sizea pt);
    \filldraw[xshift=-\sscale*\x cm, yshift=-\sscale*\y cm, black, line width=0.8pt, fill=Cerulean] (0,0) circle (\sizea pt);
}
\foreach \x / \y in {
0.536/0.928,0.928/0.536,0.000/1.072,1.072/0.000,1.268/0.732,1.464/0.000,0.732/1.268,0.000/1.464,2.000/1.155,1.155/2.000,0.000/2.309,2.309/0.000,0.732/2.732,2.732/0.732,2.000/2.000,2.928/0.000,1.464/2.536,2.536/1.464,0.000/2.928
}{
    \filldraw[xshift=\sscale*\x cm, yshift=\sscale*\y cm, black, line width=0.8pt, fill=LimeGreen] (0,0) circle (\sizeb pt);
    \filldraw[xshift=\sscale*\x cm, yshift=-\sscale*\y cm, black, line width=0.8pt, fill=LimeGreen] (0,0) circle (\sizeb pt);
    \filldraw[xshift=-\sscale*\x cm, yshift=\sscale*\y cm, black, line width=0.8pt, fill=LimeGreen] (0,0) circle (\sizeb pt);
    \filldraw[xshift=-\sscale*\x cm, yshift=-\sscale*\y cm, black, line width=0.8pt, fill=LimeGreen] (0,0) circle (\sizeb pt);
}
\foreach \x / \y in {
0.536/2.000,2.000/0.536,1.464/1.464
}{
    \filldraw[xshift=\sscale*\x cm, yshift=\sscale*\y cm, black, line width=0.8pt, fill=Goldenrod] (0,0) circle (\sizec pt);
    \filldraw[xshift=-\sscale*\x cm, yshift=\sscale*\y cm, black, line width=0.8pt, fill=Goldenrod] (0,0) circle (\sizec pt);
    \filldraw[xshift=\sscale*\x cm, yshift=-\sscale*\y cm, black, line width=0.8pt, fill=Goldenrod] (0,0) circle (\sizec pt);
    \filldraw[xshift=-\sscale*\x cm, yshift=-\sscale*\y cm, black, line width=0.8pt, fill=Goldenrod] (0,0) circle (\sizec pt);
}
\filldraw[black, line width=0.8pt, fill=Orange] (0,0) circle (\sized pt);
\foreach \x / \y in {
2.000/3.464,3.464/2.000,0.000/4.000,4.000/0.000
}{
    \filldraw[xshift=\sscale*\x cm, yshift=\sscale*\y cm, black, line width=0.8pt, fill=RubineRed] (0,0) circle (\sizee pt);
    \filldraw[xshift=\sscale*\x cm, yshift=-\sscale*\y cm, black, line width=0.8pt, fill=RubineRed] (0,0) circle (\sizee pt);
    \filldraw[xshift=-\sscale*\x cm, yshift=\sscale*\y cm, black, line width=0.8pt, fill=RubineRed] (0,0) circle (\sizee pt);
    \filldraw[xshift=-\sscale*\x cm, yshift=-\sscale*\y cm, black, line width=0.8pt, fill=RubineRed] (0,0) circle (\sizee pt);
}
\end{tikzpicture}
\caption{Intersection points of diagonals of a regular 12-gon according to their multiplicity.}
\label{fig:diagonals-multiplicity}
\end{figure}
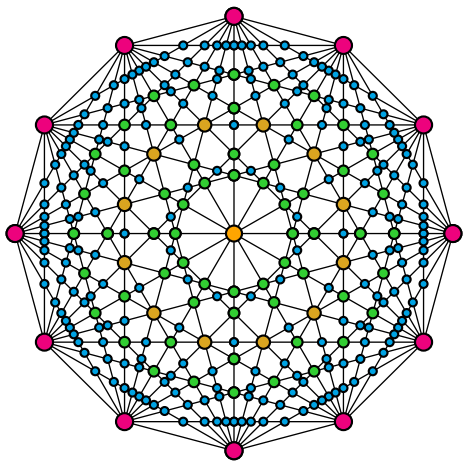

\subsection{Intersection points of diagonals}

The study of the intersection points of diagonals of regular polygons has a rich history in itself.
If the vertices of an $n$-gon are placed in generic convex position, then there would be $\binom{n}{4}$ intersection points of diagonals as every four points would be the endpoints of a unique pair of intersecting diagonals.
If the vertices are placed so that the resulting $n$-gon is regular, then the number of intersection points may be less than $\binom{n}{4}$.
Indeed, if $n$ is even, then the center of the $n$-gon is in $n/2$ diagonals.
In 1958, Steinhaus~\cite{Steinhaus1958-ry} posed the problem of showing that if $n$ is prime, then no three diagonals have a common intersection point, which was then shown by Croft and Fowler~\cite{Croft1961-be} in 1961.
In the next year, Heineken~\cite{Heineken1962-oe} proved that the same holds for $n$ odd.
Quite a while later, a remarkable explicit formula was found by Poonen and Rubinstein~\cite{Poonen1998-sp}, in 1998.
They proved that the number of intersections points of diagonals of a regular $n$-gon is given by
\begin{align*}
    \binom{n}{4}
    &- \p[\Big]{\tfrac{5n^3 - 45n^2 + 70n - 24}{24}} \ind_{2 \divides n}
    - \p[\Big]{\tfrac{3 n}{2}} \ind_{4 \divides n} \\
    &- \p[\Big]{\tfrac{45 n^2 - 262 n}{6}} \ind_{6 \divides n}
    + \p[\big]{42 n} \ind_{12 \divides n}
    + \p[\big]{60 n} \ind_{18 \divides n} \\
    &+ \p[\big]{35 n} \ind_{24 \divides n}
    - \p[\big]{38 n} \ind_{30 \divides n}
    - \p[\big]{82 n} \ind_{42 \divides n}
    - \p[\big]{330 n} \ind_{60 \divides n} \\
    &- \p[\big]{144 n} \ind_{84 \divides n}
    - \p[\big]{96 n} \ind_{90 \divides n}
    - \p[\big]{144 n} \ind_{120 \divides n}
    - \p[\big]{96 n} \ind_{210 \divides n},
\end{align*}
where $\ind_{k \divides n}$ is equal to $1$ if $k$ divides $n$ and $0$ otherwise.
This corrects an early and forgotten result of Bol~\cite{Bol1936-jn} from 1936, see~\cite{Poonen1998-sp} for more historical remarks.
While the expression above is not quite a polynomial, it becomes a polynomial when restricted to each residue class modulo $2520$.
Poonen and Rubinstein also obtained a similar expression for the number of regions formed by the diagonals formed by a regular $n$-gon, and for the number of points that are intersection points of exactly three diagonals.

It can also be read immediately from this formula that the limit \eqref{eq:karamata-limit} holds if the points are counted without multiplicity, since the number of them that occur with higher multiplicity are of order smaller than $n^4$.

\subsection{Extending Karamata's result}

What happens if instead of drawing the diagonals as segments between the vertices of a regular $n$-gon, we draw the lines between the vertices.
That is, extend the diagonals to infinite lines in the plane.
Now, each pair of lines intersect unless they are parallel.
This intersection point do not need to be contained in the interior of the $n$-gon, but it could be anywhere in the plane.
See \Cref{fig:points-outside} for an example.
Which proportion of these intersection points now lie inside a disc or radius $r$ centered at the origin, for all $r \geq 0$?

If $r = 1$, then the answer must be roughly $1/3$, as given any four vertices of the $n$-gon, from the three possible pairings of the diagonals, only one pair intersects inside the disc.
It is reasonable to expect then the answer to be
\begin{equation*}
    \frac{2}{\pi^2} \Li_2(r^2)
\end{equation*}
for $0 \leq r \leq 1$.
While that is true, this gives us no real clue to what happens when $r > 1$, since the dilogarithm function has no analytic extension to $\RR$.

\begin{figure}[!ht]
\centering
\begin{tikzpicture}
\pgfmathsetmacro{\sscale}{0.75}
\definecolor{col_a}{RGB}{230,97,1};
\definecolor{col_b}{RGB}{253,184,99};
\definecolor{col_c}{RGB}{178,171,210};
\definecolor{col_d}{RGB}{94,60,153};
\tkzSetUpStyle[color=black,line width=0.5pt,add = 5 and 5]{lin}
\clip (-5,-5) rectangle (5,5);
\tkzDefPoints{0/0/O, \sscale*4/0/X1}
\tkzDefPointsBy[rotation=center O angle 360/11](X1){X2}
\tkzDefPointsBy[rotation=center O angle 360/11](X2){X3}
\tkzDefPointsBy[rotation=center O angle 360/11](X3){X4}
\tkzDefPointsBy[rotation=center O angle 360/11](X4){X5}
\tkzDefPointsBy[rotation=center O angle 360/11](X5){X6}
\tkzDefPointsBy[rotation=center O angle 360/11](X6){X7}
\tkzDefPointsBy[rotation=center O angle 360/11](X7){X8}
\tkzDefPointsBy[rotation=center O angle 360/11](X8){X9}
\tkzDefPointsBy[rotation=center O angle 360/11](X9){X10}
\tkzDefPointsBy[rotation=center O angle 360/11](X10){X11}
\tkzDrawSegments[lin](X1,X2 X1,X3 X1,X4 X1,X5 X1,X6 X1,X7 X1,X8 X1,X9 X1,X10 X1,X11)
\tkzDrawSegments[lin](X2,X3 X2,X4 X2,X5 X2,X6 X2,X7 X2,X8 X2,X9 X2,X10 X2,X11)
\tkzDrawSegments[lin](X3,X4 X3,X5 X3,X6 X3,X7 X3,X8 X3,X9 X3,X10 X3,X11)
\tkzDrawSegments[lin](X4,X5 X4,X6 X4,X7 X4,X8 X4,X9 X4,X10 X4,X11)
\tkzDrawSegments[lin](X5,X6 X5,X7 X5,X8 X5,X9 X5,X10 X5,X11)
\tkzDrawSegments[lin](X6,X7 X6,X8 X6,X9 X6,X10 X6,X11)
\tkzDrawSegments[lin](X7,X8 X7,X9 X7,X10 X7,X11)
\tkzDrawSegments[lin](X8,X9 X8,X10 X8,X11)
\tkzDrawSegments[lin](X9,X10 X9,X11)
\tkzDrawSegments[lin](X10,X11)
\foreach \x / \y in {
4.691/2.353,3.365/9.440,-7.107/3.959,2.674/0.191,-0.569/-1.597,-1.293/-2.832,1.662/-3.639,-2.619/3.023,2.730/0.000,-3.838/1.127,-2.242/1.356,-6.176/4.765,4.930/-3.169,5.943/-3.023,0.499/0.321,-1.732/-0.000,-4.850/10.621,-1.984/5.186,2.353/-1.285,-3.838/-1.127,2.988/3.448,-3.838/-1.127,0.096/4.995,-4.323/4.499,5.415/9.145,-4.504/-2.163,-1.984/0.860,-2.273/-9.760,4.930/-8.725,3.128/-1.356,4.089/3.289,3.365/10.081,4.930/14.771,1.662/-3.639,9.500/-4.765,-4.850/4.575,4.584/-1.987,2.146/0.266,1.662/-3.639,2.533/-0.670,-2.908/-7.598,2.493/-0.806,-1.788/-5.855,0.246/0.540,-4.850/-2.004,-0.389/-3.344,-1.401/1.127,-2.273/-2.483,-2.619/0.769,-0.569/3.959,-8.119/7.788,2.533/2.282,-3.838/3.579,-2.619/1.302,-3.838/6.792,-2.950/1.896,-1.024/5.506,-14.392/-17.549,-3.838/-2.848,6.634/0.670,3.893/-3.959,-4.768/-0.321,-6.554/-9.145,-2.273/9.760,-0.569/3.959,9.349/-3.023,2.493/0.806,-5.888/-8.109,2.781/4.150,0.937/2.512,0.066/1.797,-5.888/2.063,-0.042/-2.163,5.596/-5.435,-4.958/0.616,-1.401/-1.127,10.915/-13.910,12.618/2.063,-3.838/22.369,-3.507/0.000,-11.676/0.000,0.066/3.409,-21.942/-5.801,15.484/17.870,-2.619/7.149,5.703/-5.801,-2.619/-3.023,0.499/-3.471,5.415/5.353,7.934/1.797,-5.164/0.936,-7.189/-6.982,4.377/1.285,1.284/2.353,1.315/-9.145,3.712/-3.344,-7.772/-2.282,2.297/-1.476,8.281/6.662,12.618/-2.063,-1.204/-6.122,0.096/2.924,2.699/1.127,8.865/2.603,-4.768/0.321,-11.526/-5.535,-2.950/-1.896,-1.984/-2.733,-12.053/-13.910,-1.607/1.448,-2.242/-1.356,-1.153/2.353,-4.562/0.000,5.596/-2.483,-15.322/16.743,2.533/-2.282,1.662/5.892,7.119/-2.702,1.134/-1.842,12.133/-9.760,-3.838/-1.127,3.365/-2.163,9.984/-20.381,2.245/1.651,1.662/-7.964,0.897/-6.241,-2.619/3.023,-0.569/-0.657,0.499/-0.321,0.649/-0.191,-1.024/3.252,-0.834/1.476,-0.569/-3.959,1.662/-2.407,-17.660/-0.860,-1.293/1.493,-14.392/-5.947,-6.704/-2.436,-3.838/12.196,-1.153/0.741,-1.689/-2.217,4.930/-6.471,-0.569/-3.959,3.893/3.959,5.943/3.023,-3.203/1.036,-0.124/-0.860,-1.293/2.832,-4.504/2.163,-1.607/-0.806,-2.619/-7.149,22.687/0.616,3.365/-2.163,1.315/2.457,-5.305/-3.409,-0.084/0.587,0.996/-2.603,-2.423/-2.353,9.984/20.381,-7.772/7.249,3.712/-10.621,3.365/0.091,-0.569/-0.366,-6.176/-2.512,-3.114/0.000,2.435/-1.356,-1.954/1.987,2.674/4.516,-0.569/-2.728,0.897/-1.036,0.361/4.765,3.365/-2.163,4.930/-1.448,-3.838/-4.429,2.781/1.896,-2.718/-0.616,-8.119/15.707,0.731/-2.832,-5.888/8.109,-6.069/1.448,-0.281/1.955,-15.322/-8.824,-2.619/-0.071,-2.619/5.736,-3.838/-12.196,3.365/-2.163,9.984/5.186,6.050/0.936,1.134/1.309,4.584/3.218,4.485/1.651,-7.592/-0.587,2.988/-0.877,3.365/4.416,-6.069/-4.598,0.499/-7.598,-3.838/1.127,-1.984/2.733,-14.392/17.549,-2.619/4.254,0.443/5.535,-1.500/-2.512,-2.135/2.603,-1.788/-2.903,-1.689/2.217,1.315/-2.457,8.281/-6.662,1.662/11.557,1.898/-2.832,1.315/1.226,-14.392/10.272,2.880/-3.814,-3.838/-4.079,-3.838/4.429,1.134/-5.435,7.646/1.547,-0.389/-0.448,4.930/1.448,1.662/3.639,-2.619/-0.570,-4.185/2.308,0.195/1.356,-0.389/-10.621,-2.619/-3.023,-0.569/10.005,9.500/2.512,16.703/5.801,-3.838/-1.127,3.128/1.356,-9.822/1.987,-1.054/2.308,1.662/-0.336,-2.619/-18.219,3.365/3.883,-0.569/-3.959,-1.788/-2.063,15.484/9.951,-6.554/16.422,18.753/12.783,-0.569/-15.561,-2.619/-22.544,-2.135/1.372,0.096/-0.670,-10.457/1.896,0.443/2.583,6.231/-0.321,-1.607/-1.448,0.731/5.086,3.365/0.988,-8.407/-5.086,6.634/-7.249,-15.322/2.778,-3.838/-1.127,-5.541/-2.603,1.284/-2.353,-5.541/0.349,-21.595/-6.982,4.930/-14.771,7.119/10.621,9.500/-2.512,-1.895/0.556,-9.822/0.266,-1.689/4.471,-0.569/-5.572,8.865/20.893,-7.107/-3.959,-0.569/-3.959,-4.473/-3.289,4.485/-1.651,-6.901/-3.639,5.068/3.639,2.146/-1.606,-0.569/-3.959,-2.619/-3.023,2.146/-1.987,-2.619/-3.023,-0.569/-3.959,-2.619/0.952,-5.623/-1.651,1.315/-1.226,1.134/2.483,-0.569/1.247,-0.569/-3.959,7.646/-16.743,0.731/-5.086,1.769/1.933,-0.569/7.110,-3.838/3.198,3.365/2.163,-0.569/1.705,7.934/-1.797,-8.119/-7.788,-0.192/-5.245,3.712/10.621,11.203/3.289,-4.850/-2.702,1.662/-3.639,-8.892/12.783,2.549/6.662,-3.838/-3.579,-8.119/0.511,-2.135/-2.603,-3.838/1.127,-1.500/-4.765,-3.838/-6.792,12.215/12.783,-0.084/-0.587,5.861/0.000,7.646/2.778,0.096/1.693,-2.135/-0.349,0.869/0.000,6.231/3.471,-25.876/7.598,-0.569/3.318,-7.772/-0.670,1.662/-7.431,-17.660/-20.381,-7.772/2.282,-6.176/-4.765,1.975/0.000,-1.204/1.797,-0.569/-11.236,2.699/-1.127,-0.281/-1.955,2.297/1.476,7.646/-2.778,1.134/-1.309,0.302/-2.603,-1.457/-6.982,7.119/0.448,0.443/0.511,-4.323/-8.824,-1.500/4.765,-0.569/-2.087,-2.718/0.616,-3.491/4.379,-2.619/-0.952,0.443/1.742,-0.281/-6.662,2.674/-0.191,5.746/4.226,5.596/2.483,-2.619/7.348,-14.392/5.947,-9.822/-21.508,1.662/7.431,-0.222/-2.778,7.646/-8.824,3.365/-3.394,3.365/-10.081,1.662/-11.557,-10.457/-4.150,-15.322/-16.743,2.353/1.285,3.365/5.756,1.662/3.639,1.580/0.616,0.821/-1.797,-5.434/-1.356,-3.074/1.476,6.814/4.379,0.593/-0.000,3.365/-2.163,5.596/0.229,0.443/-7.788,-3.838/-1.127,15.484/-1.651,1.216/1.181,-14.392/-10.272,-6.704/-3.610,-2.619/-1.302,-0.569/-3.959,1.662/1.068,4.000/0.000,-5.305/-1.797,1.662/7.964,-23.645/0.000,1.662/-1.918,-0.916/-1.547,-1.689/-4.471,7.449/-1.575,6.634/-0.670,5.068/-3.639,19.418/-11.748,1.662/-0.488,22.687/-14.580,2.435/-0.715,2.146/6.312,-0.916/1.547,-6.069/0.806,9.500/-8.558,-1.293/-1.493,1.662/17.603,-3.838/-26.694,-4.473/3.289,1.662/-1.385,4.666/-1.036,-0.569/-3.959,3.365/-17.358,2.245/-1.651,1.769/-1.933,8.865/-6.928,0.246/-0.540,-6.554/-16.422,4.930/8.725,-0.281/0.616,5.596/9.760,-1.895/1.896,0.443/7.788,-8.119/5.535,2.146/1.987,-0.569/-3.959,6.231/7.598,4.930/-0.806,1.315/-3.099,5.219/-1.896,-6.554/1.226,-2.135/-1.372,-1.607/-6.471,-2.423/2.353,0.384/1.651,1.134/5.435,1.481/-3.023,0.996/2.603,-2.619/-3.023,-0.222/2.778,-1.895/-4.150,26.968/-0.000,-0.281/-0.616,-1.342/-1.036,-0.192/-2.674,-3.688/-6.662,-2.619/1.683,-17.660/20.381,-2.273/4.204,-4.323/5.673,4.691/-2.353,-1.024/-3.252,-2.512/-0.936,3.365/-2.163,-3.838/-15.091,0.096/0.670,1.134/9.760,-8.892/-12.783,3.365/-2.163,0.066/6.122,-6.069/-8.725,22.687/14.580,-7.772/0.670,3.365/-2.163,-12.688/-1.575,3.365/-5.114,0.096/-1.693,0.499/3.471,6.050/1.776,0.649/4.516,-0.389/2.702,1.662/0.336,8.865/-20.893,10.915/13.910,-2.619/4.896,2.674/-2.063,-1.788/0.191,-6.704/2.436,1.662/1.385,8.100/0.000,-0.569/-3.959,5.746/-4.226,-3.838/26.694,1.769/-7.598,-2.908/0.321,-0.916/2.778,-6.957/-1.575,4.584/6.312,15.484/-9.951,0.384/-1.651,19.418/11.748,3.365/-0.290,2.880/-5.535,-25.876/-7.598,6.050/6.982,7.934/13.399,-4.850/-4.575,12.133/9.760,-7.592/0.587,-5.541/-0.349,-0.569/-6.213,0.897/1.036,5.219/-4.150,0.066/-1.797,3.365/-2.163,8.865/6.928,-1.457/-0.936,1.662/0.488,1.134/1.842,7.934/6.122,4.930/6.471,3.365/-7.368,4.584/4.059,3.365/-2.163,5.596/-9.760,-3.838/1.127,1.134/-9.760,-0.677/0.000,-1.457/-3.190,-3.838/1.127,3.365/-2.163,7.449/1.575,2.297/5.801,0.821/1.797,9.349/3.023,-2.273/-4.204,-9.338/-3.639,2.533/-4.995,-4.185/-2.308,-3.838/-3.198,-0.569/-3.959,7.269/-2.832,1.315/4.820,1.662/-3.639,-6.554/9.145,-1.293/5.086,-6.704/3.610,-0.569/-3.959,17.187/-4.150,5.219/-0.556,-3.838/1.127,-11.388/10.621,1.950/0.936,-1.457/0.936,-6.069/-0.806,0.066/-6.122,9.984/-5.186,11.203/-3.289,-1.054/-2.308,5.415/-1.226,5.703/5.801,1.134/3.181,0.499/7.598,3.365/-23.404,2.674/-4.516,-17.660/0.860,-0.569/-1.705,2.116/-5.186,-3.310/0.670,0.361/2.512,7.646/8.824,1.134/-2.483,1.027/-1.476,-6.554/-1.226,1.662/-3.639,0.443/-0.511,0.897/6.241,1.662/1.918,0.649/0.191,1.315/-0.386,6.050/-6.982,-0.569/-10.005,-5.305/6.122,-1.788/2.903,9.500/8.558,-1.788/-0.191,-1.954/-1.987,22.687/-0.616,9.984/-0.860,1.481/3.023,-10.457/4.150,-2.273/2.483,2.435/0.715,22.687/-6.662,3.712/3.344,-12.688/1.575,-1.607/6.471,0.649/-4.516,-12.053/13.910,-3.838/9.046,3.365/5.114,-2.619/-4.254,11.203/24.531,3.365/-9.440,2.435/1.356,-0.766/3.289,3.365/7.368,-19.891/-12.783,-1.204/6.122,-2.908/-0.321,-7.772/4.995,-2.619/-10.942,-5.434/1.356,-2.619/-3.023,-9.822/-4.059,3.365/0.290,1.662/0.687,-4.958/-0.616,-9.822/4.059,0.996/-1.372,-5.541/2.603,-1.607/-2.146,3.365/-0.988,-2.619/-10.300,-3.838/2.848,-2.908/7.598,-3.147/-1.226,-0.834/5.801,0.996/-4.674,4.528/6.122,-15.322/8.824,1.662/-3.639,-3.074/-1.476,-0.192/2.674,-9.239/-0.000,6.231/0.321,2.988/0.877,1.662/-3.639,-9.822/-1.987,-0.569/1.597,3.365/-2.163,-0.569/0.167,-3.384/-0.420,-3.838/15.091,7.646/16.743,3.365/12.335,1.027/1.476,18.753/-12.783,-3.838/-8.404,-0.569/-1.247,-7.189/6.982,9.984/0.860,-2.619/10.300,7.269/2.832,1.662/-3.639,1.662/-17.603,1.315/9.145,-2.537/0.000,-0.569/11.236,0.937/-2.512,2.988/1.575,-3.838/4.079,3.365/-2.163,-7.772/-7.249,0.731/2.832,-3.838/1.127,3.365/-6.289,6.231/-7.598,2.435/5.331,-6.554/4.820,1.662/-3.639,-3.838/2.467,-2.165/-1.476,-2.619/3.023,-15.322/-2.778,1.662/-1.068,5.596/5.435,2.297/-5.801,-3.203/-1.036,-0.569/2.620,2.880/3.814,1.662/-3.639,6.050/-1.776,-6.069/-1.448,3.365/-2.163,0.302/2.603,0.443/-3.082,-2.619/-7.348,0.361/-0.791,4.584/-3.218,4.930/3.169,-4.323/8.824,11.731/5.086,11.203/-24.531,-0.389/10.621,-0.834/-5.801,-6.901/3.639,-0.192/-3.372,3.365/6.289,-8.119/-0.511,-1.984/-0.860,2.781/-0.175,-2.619/22.544,-3.838/4.919,0.246/-12.783,0.897/3.289,7.934/-13.399,4.584/-4.059,3.365/-12.335,-0.124/0.860,6.050/-0.936,1.662/-3.639,-0.569/15.561,8.865/-2.603,0.443/3.082,-2.619/-0.769,2.880/-1.742,-11.041/2.163,-17.660/5.186,0.897/-3.289,4.930/0.806,-4.670/3.959,4.129/4.765,22.687/6.662,-1.153/-8.018,-1.689/0.146,-0.569/-0.167,3.365/17.358,2.039/-2.353,3.365/3.394,-11.041/-2.163,-1.500/2.512,13.919/6.982,0.246/1.714,0.649/8.109,-1.024/-2.412,3.365/-5.756,2.146/-0.266,2.146/-6.312,-3.310/-0.670,4.584/-6.312,-9.822/8.185,1.662/-3.639,-0.042/2.163,-0.569/-3.959,-2.135/4.674,1.315/0.386,-3.838/-1.127,-0.569/6.213,-2.619/0.071,1.898/2.832,0.361/-2.512,4.832/2.832,5.415/1.226,5.219/4.150,1.662/-4.978,-2.619/18.219,1.662/-3.639,-6.554/-4.820,1.662/-2.026,-0.569/2.728,-0.569/-3.318,-0.192/3.372,-3.384/0.420,-0.916/-2.778,-9.822/-8.185,-3.838/-9.046,12.215/-12.783,1.315/3.099,-3.838/-22.369,-6.069/4.598,-4.670/-3.959,7.119/2.702,-0.281/6.662,-1.204/-1.797,-6.957/1.575,15.484/-17.870,-2.997/-4.308,-2.619/-5.736,-4.850/2.004,3.365/-0.091,1.662/-2.940,5.415/-3.099,11.203/1.036,-5.305/3.409,3.365/-0.550,2.039/2.353,5.415/-5.353,-4.850/-10.621,-5.139/0.000,1.580/-0.616,-2.273/-1.842,-1.788/5.855,-2.165/1.476,0.649/-8.109,-0.569/3.959,-2.619/-3.023,4.832/-2.832,9.500/4.765,1.662/-3.639,13.919/-6.982,3.365/23.404,3.365/-2.163,-4.323/2.778,17.187/4.150,5.415/-9.145,3.365/0.550,5.415/3.099,1.216/-1.181,-2.619/-1.683,-1.457/3.190,-2.135/0.349,-3.491/-4.379,-21.595/6.982,-0.766/-3.289,1.662/-9.685,-8.119/-15.707,0.066/-3.409,-2.273/1.842,-5.001/-5.086,-0.569/-7.110,15.484/1.651,0.443/-2.583,0.361/-4.765,-0.834/0.245,1.662/-3.639,-1.895/4.150,1.662/9.685,-2.619/10.942,-3.838/8.404,-3.838/-2.467,11.731/-5.086,2.781/0.175,-3.838/-5.452,-0.834/-0.245,-9.822/-0.266,2.880/0.511,16.703/-5.801,-17.660/-5.186,3.365/-3.883,-3.254/-5.186,1.769/7.598,11.203/-1.036,-3.838/1.127,-3.838/-7.173,-2.997/4.308,-1.895/-0.556,1.769/-0.321,1.950/-0.936,-1.153/-2.353,-10.457/-1.896,-5.623/1.651,8.865/-4.674,-0.569/3.959,0.096/-4.995,5.219/1.896,-0.569/2.087,7.119/-0.448,-0.834/-1.476,-11.388/-10.621,6.814/-4.379,0.246/-1.714,7.646/-1.547,6.634/7.249,0.246/12.783,2.880/1.742,5.596/-0.229,2.116/5.186,-3.254/0.860,1.662/2.407,0.443/-1.742,3.365/-2.163,-1.788/2.063,-0.569/-2.620,8.865/4.674,-3.838/-1.127,-1.500/-0.791,-0.389/3.344,-0.569/0.657,-1.500/0.791,-1.457/6.982,-0.192/5.245,-6.704/10.887,-3.838/-1.127,2.781/-1.896,-21.942/5.801,-0.719/-1.575,-3.838/-1.127,-1.024/-5.506,-9.822/6.312,1.662/4.978,-4.850/2.702,-3.838/-1.127,3.365/-2.163,2.988/-3.448,7.934/-6.122,1.662/-5.892,4.666/1.036,2.880/-0.511,-0.719/1.575,-2.619/0.570,-1.607/2.146,-2.135/-4.674,-6.069/8.725,0.731/0.470,0.649/2.063,-2.512/0.936,0.731/-0.470,-5.164/-0.936,-1.607/0.806,-0.569/-3.959,2.533/0.670,-9.338/3.639,2.533/4.995,3.365/-4.416,-3.838/5.452,4.089/-3.289,-0.569/0.366,0.195/-1.356,6.231/-3.471,1.662/-3.639,1.769/0.321,-7.772/-4.995,10.568/-1.127,-4.323/-5.673,-5.305/1.797,2.297/2.650,7.119/-10.621,0.649/-2.063,5.219/0.556,-4.323/-2.778,-2.619/-3.023,-3.254/-0.860,-1.370/0.000,-1.153/8.018,10.568/1.127,-2.619/-3.023,-1.153/-0.741,-1.342/1.036,-1.984/-5.186,-3.688/6.662,2.146/1.606,-8.407/5.086,0.096/-2.924,-8.119/-5.535,-6.704/-10.887,-6.306/0.000,1.662/-0.687,-2.619/-3.023,-11.526/5.535,-1.024/2.412,4.528/-6.122,1.662/2.940,-2.619/-3.023,-9.822/-6.312,-0.389/-2.702,-3.147/1.226,-2.619/-3.023,2.880/5.535,4.584/1.987,0.361/0.791,0.996/4.674,0.443/-5.535,4.377/-1.285,-1.293/-5.086,2.781/-4.150,-2.619/-4.896,-3.838/1.127,-0.569/5.572,18.406/0.000,-3.254/5.186,1.134/-3.181,-9.822/21.508,2.988/-1.575,0.996/1.372,-2.619/2.183,-5.888/-2.063,-3.838/7.173,2.674/2.063,1.662/2.026,2.435/-5.331,-1.689/-0.146,-5.305/-6.122,-6.176/2.512,2.549/-6.662,-5.001/5.086,2.297/-2.650,-2.619/-2.183,-1.895/-1.896,-4.323/-4.499,1.315/-4.820,-3.838/-4.919,4.129/-4.765,-19.891/12.783,-0.389/0.448
}\filldraw[xshift=\sscale*\x cm, yshift=\sscale*\y cm, red, fill=white, line width=0.75pt] (0,0) circle (1pt);
\end{tikzpicture}
\caption{A regular 11-gon with intersection points of diagonal lines, some outside the circumcircle.}
\label{fig:points-outside}
\end{figure}
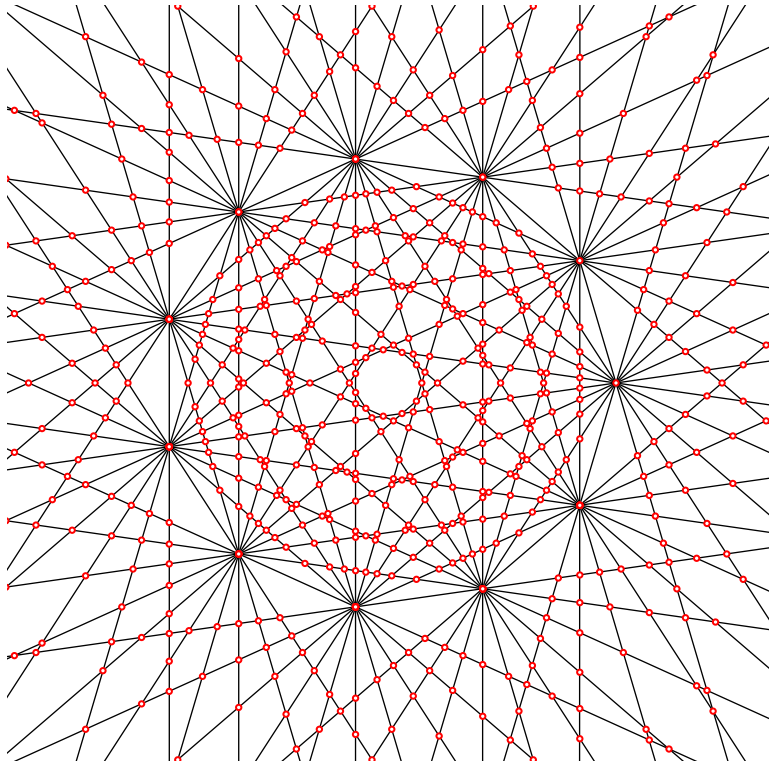

To study this problem, we actually consider an alternative formulation in the lens of \emph{geometric probability}.
Let $\cC \subseteq \RR^2$ be the unit circle centered at the origin and let $X_1$, $X_1'$, $X_2$ and $X_2'$ be independent random points chosen uniformly from $\cC$.
Let $L_1$ be the line through $X_1$ and $X_1'$ and $L_2$ be the line through $X_2$ and $X_2'$.
Then $L_1$ and $L_2$ intersect with probability $1$ in a unique point, $P$.
What is the distribution of $P$?
To make this connection with the discrete model of the $n$-gon more direct, imagine the following experiment.
Place a regular $n$-gon inscribed in the unit circle centered at the origin in the plane.
Select four vertices of $x_1$, $x_1'$, $x_2$ and $x_2'$ of the $n$-gon uniformly and independently at random.
If $n$ is large, it is indeed very likely that they are all distinct.
Now look at the lines $\ell_1$ through $x_1$ and $x_1'$, and $\ell_2$ through $x_2$ and $x_2'$.
Then $\ell_1$ and $\ell_2$ are not parallel with probability very high probability, in which case, there is a unique point $p \in \RR$ that is the intersection of $\ell_1$ and $\ell_2$.
Asking for the probability that $p$ lies in a disc of radius $0 \leq r \leq 1$ centered at the origin is then, up to lower order terms, is equivalent to computing the limit in \eqref{eq:karamata-limit}, once we correct for the fact that $p$ lies outside of the unit disc with probability approaching $2/3$.

Our main result is the following extension of Karamata's theorem.

\begin{theorem}
\label{thm:karamata}
Let $X_1$, $X_1'$, $X_2$ and $X_2'$ be points in drawn independently and uniformly from the unit circle centered at the origin.
Let $P$ be the intersection point of the line through $X_1$ and $X_1'$ with the line through $X_2$ and $X_2'$.
If we write $\ell$ for the distance of $P$ to the origin, then we have
\begin{align}
\label{eq:karamata-extended}
    \prob[\big]{\ell \leq r} = \frac{16}{\pi^3}\int_0^{\min\set{1,r}} \frac{ \arccos(t/r) \arcsin(t)}{\sqrt{1-t^2}} \dt,
\end{align}
for any $r > 0$.
In particular, for $0 \leq r \leq 1$, it holds that
\begin{align*}
    \prob[\big]{\ell \leq r} = \frac{2}{\pi^2} \Li_2(r^2).
\end{align*}
\end{theorem}

We will provide two different proofs of \Cref{thm:karamata}, one analytical and one purely geometrical.
This is in contrast with Karamata's original proof by counting, which was conceptually more involved and made use of complicated geometric constructions.

The fact that \eqref{eq:karamata-extended} simplifies to $(2 /\pi^2) \Li_2(r^2)$ for $0 \leq r \leq 1$ is proven in \Cref{sec:identity}.
This gives an unexpected way of extending $\Li_2(r)$ past the point $r \geq 1$, namely, defining it as $(\pi^2/2) \prob{\ell \leq \sqrt{r}}$.
While there is a vast literature of extensions and modifications of the dilogarithmic function (see Zagier~\cite{Zagier2007-tc} for many examples) the extension we propose seems to be new and motivated by a natural geometric problem.
In \Cref{fig:dilogextend}, we compare it with another extension of $\Li_2$ to $\RR_{\geq 0}$ that appears in certain applications~\cite{Mitchell1949-jo}.

\begin{figure}[!ht]
\centering
\begin{tikzpicture}
\definecolor{col_a}{RGB}{230,97,1};
\definecolor{col_b}{RGB}{253,184,99};
\definecolor{col_c}{RGB}{178,171,210};
\definecolor{col_d}{RGB}{94,60,153};
\begin{axis}[
    set layers,
    axis line style={on layer=axis foreground, line width = 1pt},
    width=300pt,
    height=160pt,
    axis lines=left,
    xmin=0, xmax=6,
    ymin=0, ymax=6.8,
    xtick={0, 1},
    ytick={pi*pi/6, pi*pi/2},
    yticklabels={$\frac{\pi^2}{6}$,$\frac{\pi^2}{2}$},
    legend pos=north west,
    ymajorgrids=true,
    xmajorgrids=true,
    grid style=dashed,
    legend style={legend columns=-1, nodes={scale=1}},
]

\addplot[color = col_c, dashed, line width=1.5pt]
coordinates {
(1., 1.64493) (1.05,
  1.84048) (1.1, 1.962) (1.15, 2.0548) (1.2, 2.12917) (1.25,
  2.19018) (1.3, 2.24089) (1.35, 2.28336) (1.4, 2.31907) (1.45,
  2.34913) (1.5, 2.3744) (1.55, 2.39554) (1.6, 2.41313) (1.65,
  2.42761) (1.7, 2.43935) (1.75, 2.44869) (1.8, 2.45588) (1.85,
  2.46116) (1.9, 2.46472) (1.95, 2.46675) (2., 2.4674) (2.05,
  2.4668) (2.1, 2.46506) (2.15, 2.46229) (2.2, 2.45859) (2.25,
  2.45403) (2.3, 2.44869) (2.35, 2.44264) (2.4, 2.43594) (2.45,
  2.42864) (2.5, 2.42079) (2.55, 2.41244) (2.6, 2.40362) (2.65,
  2.39437) (2.7, 2.38473) (2.75, 2.37473) (2.8, 2.36439) (2.85,
  2.35374) (2.9, 2.34281) (2.95, 2.33162) (3., 2.32018) (3.05,
  2.30852) (3.1, 2.29665) (3.15, 2.28459) (3.2, 2.27235) (3.25,
  2.25996) (3.3, 2.24741) (3.35, 2.23472) (3.4, 2.22191) (3.45,
  2.20897) (3.5, 2.19593) (3.55, 2.1828) (3.6, 2.16957) (3.65,
  2.15626) (3.7, 2.14287) (3.75, 2.12942) (3.8, 2.1159) (3.85,
  2.10232) (3.9, 2.0887) (3.95, 2.07502) (4., 2.06131) (4.05,
  2.04756) (4.1, 2.03378) (4.15, 2.01997) (4.2, 2.00613) (4.25,
  1.99227) (4.3, 1.9784) (4.35, 1.96451) (4.4, 1.95061) (4.45,
  1.9367) (4.5, 1.92278) (4.55, 1.90886) (4.6, 1.89493) (4.65,
  1.88101) (4.7, 1.86709) (4.75, 1.85318) (4.8, 1.83927) (4.85,
  1.82536) (4.9, 1.81147) (4.95, 1.79759) (5., 1.78372) (5.05,
  1.76986) (5.1, 1.75602) (5.15, 1.7422) (5.2, 1.72839) (5.25,
  1.7146) (5.3, 1.70083) (5.35, 1.68708) (5.4, 1.67335) (5.45,
  1.65964) (5.5, 1.64596) (5.55, 1.6323) (5.6, 1.61866) (5.65,
  1.60504) (5.7, 1.59146) (5.75, 1.57789) (5.8, 1.56436) (5.85,
  1.55085) (5.9, 1.53737) (5.95, 1.52391) (6., 1.51049)
};
\addlegendentry{$\int_0^r \frac{-\log |1 - t| }{t} \mathop{}\!\mathrm{d} t$}

\addplot[color = col_a, line width=1pt]
coordinates{
(0., 0.)(0.05, 0.0506393)(0.1, 0.102618)(0.15, 0.156035)(0.2,
   0.211004)(0.25, 0.267653)(0.3, 0.32613)(0.35,
  0.386606)(0.4, 0.449283)(0.45, 0.514399)(0.5,
  0.582241)(0.55, 0.653158)(0.6, 0.727586)(0.65,
  0.806083)(0.7, 0.889378)(0.75, 0.978469)(0.8, 1.07479)(0.85,
   1.18058)(0.9, 1.29971)(0.95, 1.44063)(1., 1.64493)(1.05, 1.84967) (1.1, 1.98717) (1.15,
  2.09967) (1.2, 2.19626) (1.25, 2.28133) (1.3, 2.35748) (1.35,
  2.42645) (1.4, 2.48946) (1.45, 2.54743) (1.5, 2.60109) (1.55,
  2.65098) (1.6, 2.69757) (1.65, 2.74124) (1.7, 2.78229) (1.75,
  2.82101) (1.8, 2.8576) (1.85, 2.89228) (1.9, 2.9252) (1.95,
  2.95652) (2., 2.98637) (2.05, 3.01487) (2.1, 3.04212) (2.15,
  3.0682) (2.2, 3.09321) (2.25, 3.11721) (2.3, 3.14028) (2.35,
  3.16247) (2.4, 3.18384) (2.45, 3.20445) (2.5, 3.22432) (2.55,
  3.24352) (2.6, 3.26208) (2.65, 3.28002) (2.7, 3.2974) (2.75,
  3.31423) (2.8, 3.33054) (2.85, 3.34636) (2.9, 3.36172) (2.95,
  3.37664) (3., 3.39114) (3.05, 3.40523) (3.1, 3.41894) (3.15,
  3.43229) (3.2, 3.44528) (3.25, 3.45795) (3.3, 3.47029) (3.35,
  3.48232) (3.4, 3.49406) (3.45, 3.50552) (3.5, 3.5167) (3.55,
  3.52763) (3.6, 3.5383) (3.65, 3.54874) (3.7, 3.55894) (3.75,
  3.56892) (3.8, 3.57868) (3.85, 3.58824) (3.9, 3.59759) (3.95,
  3.60675) (4., 3.61573) (4.05, 3.62452) (4.1, 3.63314) (4.15,
  3.64159) (4.2, 3.64988) (4.25, 3.65801) (4.3, 3.66599) (4.35,
  3.67382) (4.4, 3.6815) (4.45, 3.68904) (4.5, 3.69645) (4.55,
  3.70373) (4.6, 3.71088) (4.65, 3.71791) (4.7, 3.72482) (4.75,
  3.73161) (4.8, 3.73829) (4.85, 3.74485) (4.9, 3.75131) (4.95,
  3.75767) (5., 3.76392) (5.05, 3.77008) (5.1, 3.77613) (5.15,
  3.7821) (5.2, 3.78797) (5.25, 3.79375) (5.3, 3.79945) (5.35,
  3.80506) (5.4, 3.81059) (5.45, 3.81604) (5.5, 3.82141) (5.55,
  3.82671) (5.6, 3.83192) (5.65, 3.83707) (5.7, 3.84214) (5.75,
  3.84715) (5.8, 3.85208) (5.85, 3.85695) (5.9, 3.86175) (5.95,
  3.86649) (6., 3.87117)
};
\addlegendentry{$\frac{\pi^2}{2} \mathbb{P}\big( \ell \leq \sqrt{r} \big)$}

\addplot[color = black, line width=1pt]
coordinates {
(0., 0.)(0.05, 0.0506393)(0.1, 0.102618)(0.15, 0.156035)(0.2,
   0.211004)(0.25, 0.267653)(0.3, 0.32613)(0.35,
  0.386606)(0.4, 0.449283)(0.45, 0.514399)(0.5,
  0.582241)(0.55, 0.653158)(0.6, 0.727586)(0.65,
  0.806083)(0.7, 0.889378)(0.75, 0.978469)(0.8, 1.07479)(0.85,
   1.18058)(0.9, 1.29971)(0.95, 1.44063)(1., 1.64493)
};
\addlegendentry{$\Li_2(r)$}
\end{axis}
\end{tikzpicture}
\caption{Two possible extensions of the dilogarithm function to $\RR_{\geq 0}$.}
\label{fig:dilogextend}
\end{figure}
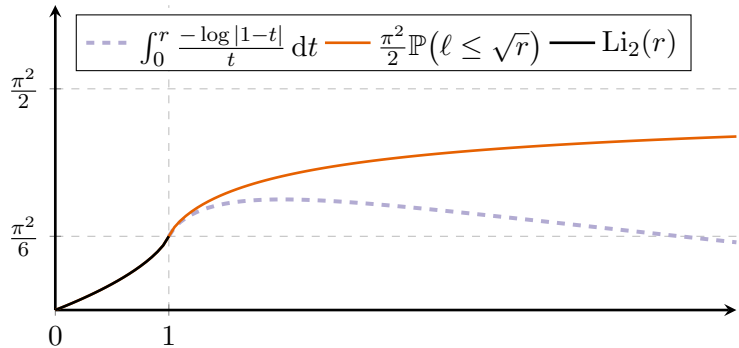

Computing the derivative $\rho(r)$ of $\prob[\big]{\ell \leq r}$ in $r$, we obtain
\begin{align}
\label{eq:def-rho}
    \rho(r) \defined \frac{\d}{\dr} \prob[\big]{\ell \leq r}
    = \begin{cases}
        \displaystyle - \frac{4 \log(1 - r^2)}{\pi^2 r} & \text{if $0 \leq r < 1$}, \\[0.8em]
        \displaystyle \frac{16}{r^2 \pi^3} \int_0^1 \frac{t \arcsin t}{\sqrt{1-t^2} \sqrt{1-t^2/r^2}} \dt & \text{if $r > 1$}.
    \end{cases}
\end{align}
We can see in \Cref{fig:diagonals123} that again the derivative explodes at $1$, showing that the function $r \mapsto \prob[\big]{\ell \leq r}$ is not analytic at $r = 1$.
The same holds for any extension of $\Li_2(r)$.

\begin{figure}[!ht]
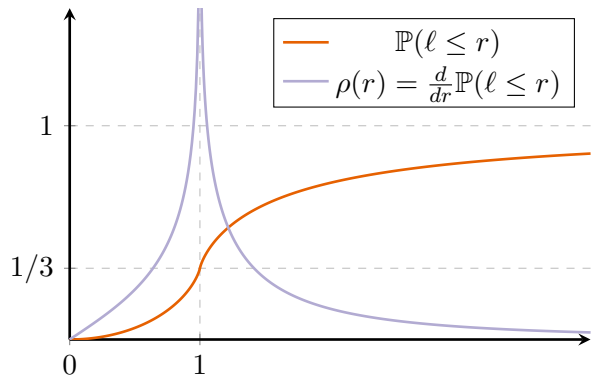

\centering

\caption{Probability that intersection point is inside disc of radius $r$, together with the density $\rho(r)$.}
\label{fig:diagonals123}
\end{figure}

\subsection{A conceptual overview}

We now explore more closely the link between the discrete and continuous models discussed above.
Denote by $\Graff_1(\RR^2)$ the set of all affine subspaces of $\RR^2$ of dimension $1$, that is, the set of all lines.
Define $\Delta \defined \set[\big]{(x,x) \st x \in \RR^2}$ and for $(x, y) \in \RR^2 \setminus \Delta$, denote by $\aff(x,y) \in \Graff_1(\RR^2)$ the \emph{affine span} of $x$ and $y$, that is, the unique line through $x$ and $y$.
Here $\Graff$ stands for \emph{affine Grassmannian}.
We equip $\Graff_1(\RR^2)$ with a topology such that the map $(x,y) \mapsto \aff(x,y)$ from  $\RR^2 \times \RR^2 \setminus \Delta$ to  $\Graff_1(\RR^2)$ is continuous.
Now take $X_1$, $X_2$, $X_3$ and $X_4$ independently distributed as $\Unif(\SS^1)$, that is, uniformly drawn from the unit circle $\SS^1 \subseteq \RR^2$.
The elements $\aff(X_1,X_2)$ and $\aff(X_3,X_4)$ are independent random lines in $\Graff_1(\RR^2)$.
In other words, each of these lines are samples from a naturally defined probability measure in $\Graff_1(\RR^2)$.
Sampling two lines independently from this measure gives parallel lines with probability zero, so they intersect in a point $\aff(X_1,X_2) \cap \aff(X_3,X_4) \in \RR^2$, inducing a probability measure in $\RR^2$.

If $\Pi_n$ is a set of $n$ equally spaced points in $\SS^1$, then we have that $\Unif(\Pi_n)$ converges in distribution to $\Unif(\SS^1)$ as $n \to \infty$.
Note that the map $\Phi \from (\RR^2)^4 \mapsto \RR^2$ that sends $(x_1,x_2,x_3,x_4) \mapsto \aff(x_1,x_2) \cap \aff(x_3,x_4)$ is continuous, except on a set of measure zero.
Therefore, as an application of the Continuous Mapping Theorem, if $x_1,x_2,x_3,x_4$ are independently distributed as $\Unif(\Pi_n)$ and $X_1,X_2,X_3,X_4$ are independently distributed as $\Unif(\SS^1)$, then $\aff(x_1,x_2) \cap \aff(x_3,x_4)$ converges in distribution to $\aff(X_1,X_2) \cap \aff(X_3,X_4)$.
\begin{equation*}
\begin{tikzcd}
{x_1, x_2, x_3, x_4 \sim \Unif(\Pi_n)} \arrow[dd, "\Phi"] \arrow[rr, "n \to \infty"] &  & {X_1,X_2,X_3,X_4 \sim \Unif(\SS^1)} \arrow[dd, "\Phi"] \\
&  &\\
{\aff(x_1,x_2)\cap\aff(x_3,x_4)} \arrow[rr, "n \to \infty"] &  & {\aff(X_1,X_2)\cap\aff(X_3,X_4)}
\end{tikzcd}
\end{equation*}
In particular, \Cref{thm:karamata} implies the original result of Karamata.

\subsection{Special values}

From \Cref{thm:karamata}, we can recover an expression for the distribution of $P = \aff(X_1,X_2) \cap \aff(X_3,X_4)$ using the rotational symmetry of $P$.
Indeed, from the density function $\rho(r)$ defined in \eqref{eq:def-rho}, we can infer that for every measurable set $S \subseteq \RR^2$, if we express the indicator function of $S$ in polar coordinates $\ind_S(r,\theta)$, then
\begin{equation}
\label{eq:explicit}
    \prob[\big]{\aff(X_1,X_2)\cap\aff(X_3,X_4) \in S} = \int_0^{2\pi} \int_0^\infty  \ind_S(r,\theta) \rho(r) \dr \dtheta.
\end{equation}
Unfortunately, this is a quite complicated integral.
It is unclear whether explicit values can be computed from this formula.
Even when $S$ is a disk, where we have \eqref{eq:karamata-extended}, not many explicit values for this probability function are known.
However, from the fact that we know a few special values of the dilogarithm function $\Li_2(x)$ for $0 \leq x \leq 1$, some surprising identities can be obtained.
For instance, from the fact that
\begin{equation*}
    \Li_2\p[\big]{1/2} = \frac{\pi^2}{12} - \frac{1}{2} (\log 2)^2,
\end{equation*}
see Zagier~\cite{Zagier2007-tc} for instance, we can obtain that
\begin{equation*}
    \prob[\big]{\ell \leq 1/\sqrt{2}} = \frac{2}{\pi^2} \Li_2(1/2) = \frac{1}{6} - \p[\Big]{\frac{\log 2}{\pi}}^2.
\end{equation*}
Similar identities can be obtained from other special values of $\Li_2$, such as
\begin{align*}
    \Li_2\p[\Big]{\frac{3 - \sqrt{5}}{2}} &= \frac{\pi^2}{15} - \p[\bigg]{\log\p[\Big]{\frac{1 + \sqrt{5}}{2}} }^2, \\
    \Li_2\p[\Big]{\frac{-1 + \sqrt{5}}{2}} &= \frac{\pi^2}{10} - \p[\bigg]{\log\p[\Big]{\frac{1 + \sqrt{5}}{2}}}^2.
\end{align*}
We do not know any special value for $\prob[\big]{\ell \leq r}$ when $r > 1$.

\subsection{The Bertrand paradox}

Being a foundational concern in geometric probability, the Bertrand paradox has led to a sizable literature in mathematics, physics and philosophy (see for instance \cite{Bower1934-oq, Jaynes1973-bc, Szekely1986-wk, Marinoff1994-kr, Gyenis2015-ty}).
We also refer to the survey of Calka~\cite{Calka2019-ua} on the classical problems of geometric probability, including the Bertrand paradox.

The problem that Bertrand~\cite{Bertrand1889-wl} proposed reads as follows.
Consider a \emph{random chord} of a given unit circle, what is the probability that this chord is longer then the edge of an equilateral triangle inscribed in the circle?
In other words, what is the probability that the length of the chord is longer that $\sqrt{3}$?
The so-called paradox comes from the fact that the term \emph{random chord} can be reasonably interpreted in many different ways.
Usually, three different \emph{solutions} are proposed: a chord obtained by choosing the midpoint uniformly at random in the unit disk; a chord obtained by choosing the distance to the midpoint uniformly in $[0,1]$ together with a uniform rotation; and a chord obtained by choosing the endpoints uniformly in the unit circle.
The probabilities are then easily computable as $1/4$, $1/2$ and $1/3$ respectively, which proves that these three random chords came from three different probability measures.

Since the last proposed random chord coincides exactly with the random lines we have been considering for Karamata's problem, it is natural to wonder which distributions arises from the intersection of random lines drawn according to the other two distributions.
Following our analytic proof of \Cref{thm:karamata}, explicit expressions for $\prob[\big]{\ell \leq r}$ can be easily computed.
For $0 \leq r \leq 1$, we respectively obtain
\begin{equation*}
    \frac{r^2}{2}, \;\; \frac{3r^4}{8}, \; \text{and} \;\; \frac{2}{\pi^2} \Li_2(r^2).
\end{equation*}
While we have simple expressions when $0 \leq r \leq 1$, the general forms of $\prob[\big]{\ell \leq r}$ are more complicated and we postpone them to \Cref{cor:bertrand} in \Cref{sec:bertrand}.

Starting from any probability measure $\mu$ on $\RR_{\geq 0}$, we can produce an associated probability measure $\cG(\mu)$ on $\Graff_1(\RR^2)$ by taking the line whose closest point to the origin\footnote{When $t = 0$, we take a line through the origin by choosing a direction uniformly in the circle.} is $(t \cos\theta, t \sin\theta)$, where $\theta$ is independently and uniformly sampled from $[0,2\pi)$.
Indeed, any rotationally invariant probability measure on $\Graff_1(\RR^2)$ can be written in this way.
Our general result is the following.

\begin{theorem}
\label{thm:analytic}
Let $\mu$ be an atomless probability measure on $\RR_{\geq 0}$ and sample two independent lines $L_1$ and $L_2$ in the plane from $\cG(\mu)$.
If $\ell$ is the distance from $L_1 \cap L_2$ to the origin, then we have
\begin{equation}
\label{eq:atomless-transform}
    \prob[\big]{\ell \leq r} = \frac{4}{\pi} \int_{0}^{r} \mu\p[\big]{[0,t]}\arccos(t/r) \d\mu(t).
\end{equation}
\end{theorem}

From the proof of \Cref{thm:analytic}, we can also compute the expected value of the distance $\ell$ of the intersection point to the origin.

\begin{corollary}
\label{cor:integrability}
Sample two independent lines $L_1$ and $L_2$ from $\cG(\mu)$, where $\mu$ is an arbitrary probability measure on $\RR_{\geq 0}$, and write $\ell$ for the distance of $L_1 \cap L_2$ to the origin.
If $\mu(\set{0}) < 1$, then $\Expec \ell = \infty$.
\end{corollary}

While this may appear unexpected at first sight, it has a simple intuitive explanation: for rotationally invariant measures, it is not very unlikely to obtain almost parallel lines, whose intersection points can be made arbitrarily far off quite quickly.
Indeed, our proof of \Cref{cor:integrability} makes this intuition rigorous.

In \Cref{sec:bertrand}, we explicitly work out one more variant of \Cref{thm:karamata} we found of interest.
The lines $L_1$ and $L_2$ are such that their closest point to the origin follows a two-dimensional standard Gaussian distribution.
Equivalently, we considered $\cG(\mu)$ where $\mu$ is a Rayleigh distribution.
An explicit formula is also computable, this time in terms of modified Bessel functions of the first kind.


\section{Proofs of the main theorem}
\label{sec:main-proof}

Recall that  $X_1$, $X_1'$, $X_2$, $X_2'$ are independent random variables chosen uniformly from the unit circle $\SS^1 \subseteq \RR^2$ centered at the origin, $L_1 = \aff(X_1,X_1')$ and $L_2 = \aff(X_2, X_2')$ are the lines spanned by these variables and, with full probability, $P \in \RR^2$ is the intersection point of $L_1$ and $L_2$.

\begin{figure}[!ht]
\centering
\begin{tikzpicture}
\tkzDefPoints{0/0/O, 1.75/0/R, 2/0/R'}
\tkzDefPointOnCircle[through = center O angle 20 point R]
\tkzGetPoint{X}
\tkzDefPointOnCircle[through = center O angle 125 point R]
\tkzGetPoint{X'}
\tkzDefMidPoint(X,X')
\tkzGetPoint{A}
\tkzDrawCircle[line width=0.5pt](O,R)
\tkzDrawLine[line width=0.5pt,add = 0.5 and 0.5](X,X')
\tkzDrawSegment[line width=0.8pt,dashed,gray](O,A)
\tkzDrawSegment[line width=0.8pt,dashed,gray](O,R')
\tkzMarkRightAngle[size=0.2](O,A,X)
\tkzDrawPoints[size=3, fill=black](X,X',A)
\tkzMarkAngle[size=0.25](R,O,A)
\tkzLabelAngle[above right, pos=0.2, xshift=-0.25em, yshift=-0.2em](R,O,A){$\theta_i$}
\tkzLabelPoint[above right](X){$X_i$}
\tkzLabelPoint[above](A){$A_i$}
\tkzLabelPoint[above, xshift=-0.15em](X'){$X_i'$}
\tkzLabelLine[pos=-0.5,above](X,X'){$L_i$}
\draw [gray,decorate,decoration={brace,amplitude=5,raise=2},xshift=0,yshift=0]
(O) -- (A) node [black,midway,xshift=-1.2em,yshift=0.5em] {$\ell_i$};
\end{tikzpicture}
\caption{Line $L_i$ through points $X_i$ and $X_i'$ in the unit circle.}
\label{fig:line-definition}
\end{figure}
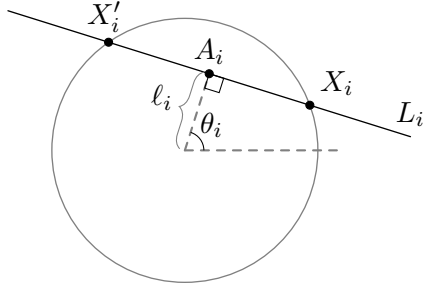

For $i = 1,2$, note that each $L_i$ can be uniquely determined by its distance $\ell_i \in [0,1]$ to the origin, and the angle $\theta_i \in [0,2\pi)$ such that the closest point in $L_i$ to the origin is $A_i \defined (\ell_i \cos \theta_i, \ell_i \sin \theta_i)$ as in \Cref{fig:line-definition}.

We now determine the law of $L_i$ in terms of its coordinates $(\ell_i, \theta_i)$.
To do so, we determine $\prob{\ell_i \leq r}$ for $0 < r < 1$.
Consider the circle $\cC_r$ of radius $r$ centered at the origin and trace the two lines through $X_i$ that are tangent to $\cC_r$.
They intersect the unit circle in points $S$ and $T$, see \Cref{fig:line-dist}.

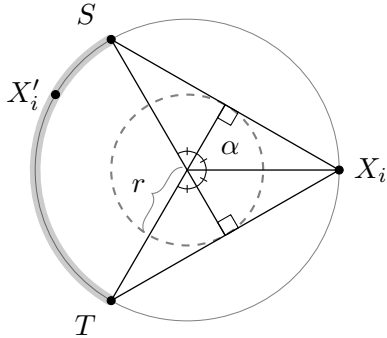
\begin{figure}[!ht]
\centering
\begin{tikzpicture}
\tkzDefPoints{0/0/O, 2/0/X, 1/0/R}
\tkzDefLine[tangent from = X](O,R)
\tkzGetPoints{A2}{A1}
\tkzInterLC(X,A1)(O,X)
\tkzGetPoints{G1}{F1}
\tkzInterLC(X,A2)(O,X)
\tkzGetPoints{F2}{G2}
\tkzDefPointOnCircle[through = center O angle 150 point X]
\tkzGetPoint{X1}
\tkzInterLC(O,F2)(O,R)
\tkzGetPoints{R1}{R2}
\tkzDrawCircle(O,X)
\tkzDrawCircle[line width=0.8pt,dashed,gray](O,R)
\tkzDrawSegments[line width=0.5pt](O,X O,F1 O,F2 O,A1 O,A2 X,F1 X,F2)
\tkzMarkAngle[size=0.25, mark=|, mksize=1.5](A1,O,F1)
\tkzMarkAngle[size=0.25, mark=|, mksize=1.5](F2,O,A2)
\tkzMarkAngle[size=0.25, mark=|, mksize=1.5](X,O,A1)
\tkzMarkAngle[size=0.25, mark=|, mksize=1.5](A2,O,X)
\tkzMarkRightAngle[size=0.2](O,A1,X)
\tkzMarkRightAngle[size=0.2](O,A2,X)
\tkzDrawPoints[size=3, fill = black](X,F1,F2,X1)
\tkzDrawArc[delta=0,line width=4,opacity=0.2](O,F1)(F2)
\tkzLabelAngle[pos=0.32,xshift=0.75em,yshift=0.35em](X,O,A1){$\alpha$}
\tkzLabelPoint[right](X){$X_i$}
\tkzLabelPoint[above left](F1){$S$}
\tkzLabelPoint[below left](F2){$T$}
\tkzLabelPoint[left](X1){$X_i'$}
\draw [gray,decorate,decoration={brace,mirror,amplitude=5,raise=2},xshift=0,yshift=0]
(O) -- (R2) node [black,midway,xshift=-1.0em,yshift=0.6em] {$r$};
\end{tikzpicture}
\caption{The distance from the origin to the line $X_i X_i'$ is less than $r$ when $X_i'$ lie in the arc $ST$.}
\label{fig:line-dist}
\end{figure}

We only have $\ell_i \leq r$ if $X_i'$ lies in the arc $ST$, which covers a $(2\pi - 4\alpha)/2\pi$ proportion of the unit circle, where $\alpha = \arccos r$.
Therefore,
\begin{equation}
\label{eq:cumulative}
    \prob[\big]{\ell_i \leq r} = 1 - \frac{2}{\pi}\arccos r  = \frac{2}{\pi} \arcsin r = \frac{2}{\pi} \int_0^r \frac{1}{\sqrt{1 - \ell_i^2}} \d\ell_i.
\end{equation}
Since the expression above does not depend on $\theta_i$, we have that $\ell_i$ and $\theta_i$ are independent.
We write the intersection $P = L_1 \cap L_2$ in polar coordinates $P = (\ell \cos\theta, \ell \sin \theta)$, and since $\theta_1$ and $\theta_2$ are uniform and independent, so is $\theta$ uniform and independent from $\ell$.

\subsection{Geometric proof}
We can assume that $\ell_1 \leq \ell_2$, since we have
\begin{equation*}
    \prob[\big]{\ell \leq r} = \prob[\big]{\ell \leq r \given \ell_1 \leq \ell_2} + \prob[\big]{\ell \leq r \given \ell_2 \leq \ell_1} = 2 \prob[\big]{\ell \leq r \given \ell_1 \leq \ell_2}
\end{equation*}
from the fact that $\prob[\big]{\ell_1 = \ell_2} = 0$.
Moreover, if $\ell \leq r$ and $\ell_1 \leq \ell_2$, then we must have that $\ell_1 \leq \min\set{1,r}$, as $\ell \geq \ell_1,\ell_2$.
In other words, we can write
\begin{align*}
    \prob[\big]{\ell \leq r}
    &= 2 \prob[\big]{\ell \leq r \given \ell_1 \leq \ell_2} \\
    &= \frac{8}{\pi^2} \int_{0}^{\min\set{1,r}} \int_{0}^{t_2} \frac{\prob[\big]{\ell \leq r \given \ell_1 = t_1, \ell_2 = t_2}}{\sqrt{1 - t_1^2} \sqrt{1 - t_2^2}} \dt_1 \dt_2.
\end{align*}
To carry our the computation, we need to determine $\prob[\big]{\ell \leq r \given \ell_1 = t_1, \ell_2 = t_2}$, which should be formally interpreted as a conditional density function.

Let $\cC_{t_2}$ be the circle of radius $t_2$ centered at the origin $O$.
Let the points $F$ and $G$ be the intersection of the line $L_1$ with $\cC_r$, as in \Cref{fig:intersec-dist}.
Let $F'$ and $F''$ be the points of tangency of the lines through $F$ that are tangent to the circle $\cC_{t_2}$.
Similarly, $G'$ and $G''$ are the tangency points of the lines through $G$ that are tangent to the circle $\cC_{t_2}$.

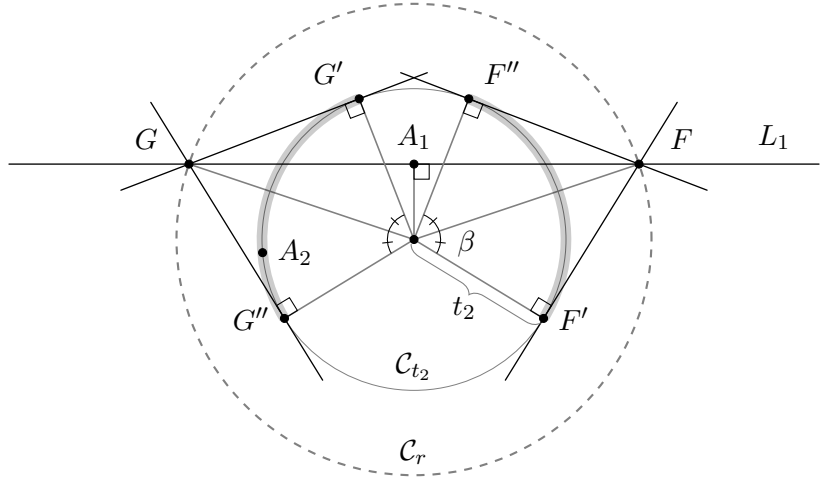
\begin{figure}[!h]
\centering
\begin{tikzpicture}
\pgfmathsetmacro{\sscale}{1.25}
\tkzDefPoints{0/0/O, 0/0.8*\sscale/A1, 2.5*\sscale/0/R, 0/1.6*\sscale/Y}
\tkzDefPoint(185:1.6*\sscale){A2}
\tkzDefLine[tangent at = A1](O)
\tkzGetPoint{H1}
\tkzInterLC(A1,H1)(O,R)
\tkzGetPoints{F1}{G1}
\tkzDefLine[tangent from = F1](O,A2)
\tkzGetPoints{F2}{F3}
\tkzDefLine[tangent from = G1](O,A2)
\tkzGetPoints{G2}{G3}
\tkzDrawCircle(O,A2)
\tkzDrawCircle[line width=0.8pt,dashed,gray](O,R)
\tkzDrawLine[line width=0.5pt,add= 0.4 and 0.4](F1,G1)
\tkzDrawLine[line width=0.5pt,add= 0.4 and 0.4](F1,F2)
\tkzDrawLine[line width=0.5pt,add= 0.4 and 0.4](F1,F3)
\tkzDrawLine[line width=0.5pt,add= 0.4 and 0.4](G1,G2)
\tkzDrawLine[line width=0.5pt,add= 0.4 and 0.4](G1,G3)
\tkzDrawSegments[line width=0.6pt,gray](O,A1 O,F1 O,F2 O,F3 O,G1 O,G2 O,G3)
\tkzMarkRightAngle[size=0.2](O,A1,F1)
\tkzMarkAngle[size=0.35,mark=|,mksize=2](F1,O,F3)
\tkzMarkAngle[size=0.35,mark=|,mksize=2](F2,O,F1)
\tkzMarkAngle[size=0.35,mark=|,mksize=2](G1,O,G3)
\tkzMarkAngle[size=0.35,mark=|,mksize=2](G2,O,G1)
\tkzMarkRightAngle[size=0.2](O,F2,F1)
\tkzMarkRightAngle[size=0.2](O,F3,F1)
\tkzMarkRightAngle[size=0.2](O,G2,G1)
\tkzMarkRightAngle[size=0.2](O,G3,G1)
\tkzDrawPoints[size=3, fill=black](O,A1,A2,F1,F2,F3,G1,G2,G3)
\tkzDrawArc[delta=0,line width=4,opacity=0.2](O,F2)(F3)
\tkzDrawArc[delta=0,line width=4,opacity=0.2](O,G2)(G3)
\tkzLabelAngle[pos=0.5,xshift=0.5em](F2,O,F1){$\beta$}
\tkzLabelPoint[above](A1){$A_1$}
\tkzLabelPoint[right](A2){$A_2$}
\tkzLabelPoint[above right,xshift=6](F1){$F$}
\tkzLabelPoint[above left,xshift=-6](G1){$G$}
\tkzLabelPoint[right](F2){$F'$}
\tkzLabelPoint[above right](F3){$F''$}
\tkzLabelPoint[above left](G2){$G'$}
\tkzLabelPoint[left](G3){$G''$}
\tkzLabelLine[pos=1.3,above](G1,F1){$L_1$}
\tkzLabelCircle[below, yshift=-0.8em](O,A1)(180){$\cC_{t_2}$}
\tkzLabelCircle[below, yshift=-3.8em](O,A1)(180){$\cC_r$}
\draw [gray,decorate,decoration={brace,mirror,amplitude=5,raise=2},xshift=0,yshift=0]
(O) -- (F2) node [black,midway,xshift=-0.5em,yshift=-1em] {$t_2$};
\end{tikzpicture}
\caption{$A_2$ must lie in the highlighted arcs for $P = L_1 \cap L_2$ to lie inside the (dashed) circle $\cC_r$.}
\label{fig:intersec-dist}
\end{figure}

If $P = L_1 \cap L_2$ lies inside $C_r$, then $P$ must be in the segment $GF$.
This is only the case if $A_2$ lie in the arcs $G'G''$ or $F'F''$.
Indeed, recall that $A_2$ is the point in $L_2$ that is closest to $O$, so it must lie in the circle $\cC_{t_2}$.
The probability that $\ell \leq r$ when $\ell_1 = t_1$ and $\ell_2 = t_2$ is then $4\beta/2\pi$, where $\beta = \angle F O F' = \arccos(t_2/r)$.
In other words,
\begin{equation*}
    \prob[\big]{\ell \leq r \given \ell_1 = t_1, \ell_2 = t_2} = \frac{2 \arccos(t_2/r)}{\pi}.
\end{equation*}
Finally, this gives
\begin{align*}
    \prob[\big]{\ell \leq r}
    &= \frac{8}{\pi^2} \int_{0}^{\min\set{1,r}} \int_{0}^{t_2} \frac{\prob[\big]{\ell \leq r \given \ell_1 = t_1, \ell_2 = t_2}}{\sqrt{1 - t_1^2} \sqrt{1 - t_2^2}} \dt_1 \dt_2 \\
    &= \frac{16}{\pi^3} \int_{0}^{\min\set{1,r}} \int_{0}^{t_2} \frac{\arccos(t_2/r)}{\sqrt{1 - t_1^2} \sqrt{1 - t_2^2}} \dt_1 \dt_2 \\
    &= \frac{16}{\pi^3} \int_{0}^{\min\set{1,r}} \frac{\arccos(t_2/r) \arcsin(t_2)}{\sqrt{1 - t_2^2}} \dt_2,
\end{align*}
which completes the proof of \Cref{thm:karamata}.
\qed

\subsection{Analytic proof}
\label{sec:analytic}

For the benefit of further discussion, we will make our analytic proof in higher generality and prove \Cref{thm:analytic}.
Indeed, we take the lines $L_1$ and $L_2$ to be independently drawn $\cG(\mu)$, where $\mu$ is an \emph{atomless} probability measure on $\RR_{\geq 0}$.
In other words, parametrise the lines $L_i$ by their closest point $P_i$ to the origin, written in polar coordinates $P_i = (\ell_i \cos \theta_i, \ell_i \sin \theta_i)$, where $\theta_i$ are independently and uniformly distributed in $[0, 2\pi)$ and that $\ell_i$ are independently sampled from $\mu$.
We write $\ell$ for the distance of $L_1 \cap L_2$ to the origin and our goal is to compute $\prob[\big]{\ell \leq r}$ to obtain \eqref{eq:atomless-transform}.

\begin{proof}[Proof of \Cref{thm:analytic}]
Notice that $\ell \leq \ell_1, \ell_2$, so we may assume $\ell_1, \ell_2 \leq r$ and since $\theta_1$, $\theta_2$, $\ell_1$ and $\ell_2$ are all independent, we may assume that $\ell_1 \leq \ell_2$ and $\theta_1 - \theta_2 \in [0,\pi]$ by the means of conditioning.
In other words, we use the fact that
\begin{align*}
    \prob[\big]{\ell \leq r} &= 2 \prob[\big]{\ell \leq r \given \ell_1 \leq \ell_2\leq r} \\
    &= 4 \prob[\big]{\ell \leq r \given \ell_1 \leq \ell_2 \leq r, \theta_1 - \theta_2 \in [0,\pi]}
\end{align*}
and note that in the first equality, we have used the fact that $\mu$ is atomless.
We now proceed to compute the conditional probability above.

Observe that the lines $L_1$ and $L_2$ are given by
\begin{align*}
    L_1 &= \set[\big]{ (x,y) \in \RR^2 \st x \sin\theta_1 + y \cos\theta_1 = \ell_1}, \\
    L_2 &= \set[\big]{ (x,y) \in \RR^2 \st x \sin\theta_2 + y \cos\theta_2 = \ell_2 }.
\end{align*}
Since we are interested in the distance from $P = L_1 \cap L_2$ to the origin, we can instead consider the distance of $P'= L_1' \cap L_2'$ from the origin, where $L_1'$ and $L_2'$ are obtained by rotation $L_1$ and $L_2$ by an angle of $-\theta_1$.
Thus we have,
\begin{align*}
    L_1' &= \set[\big]{ (x,y) \in \RR^2 \st y = \ell_1}, \\
    L_2' &= \set[\big]{ (x,y) \in \RR^2 \st x \sin(\theta_2-\theta_1) + y \cos(\theta_2-\theta_1) = \ell_2 },
\end{align*}
which gives the intersection point
\begin{equation*}
    P' = \p[\bigg]{ \frac{\ell_1 \cos(\theta_2-\theta_1) - \ell_1}{\sin(\theta_2-\theta_1)}, \ell_2}.
\end{equation*}
Therefore, the distance from $P$ to the origin is given by
\begin{equation}
\label{eq:distance}
    \ell^2
    = \p[\bigg]{ \frac{\ell_1 \cos(\theta_2-\theta_1)  - \ell_1}{\sin(\theta_2-\theta_1) }}^2 +  \ell_2^2
    = \frac{\ell_1^2 + \ell_2^2 - 2\ell_1 \ell_2 \cos(\theta_2-\theta_1) }{\p[\big]{\sin(\theta_2-\theta_1)}^2}.
\end{equation}
We want to understand which range of values of $\theta_1 - \theta_2$ gives us $\ell^2 \leq r^2$, that is
\begin{equation}
\label{eq:cos-ineq}
    r^2\p[\big]{\cos(\theta_1 - \theta_2)}^2 - 2\ell_1 \ell_2 \cos(\theta_1 - \theta_2) - r^2 + \ell_1^2 + \ell_2^2 \leq 0.
\end{equation}
Observe that this is a quadratic inequality in $\cos(\theta_1 - \theta_2)$, and that the roots of the associated quadratic polynomial are given by
\begin{equation*}
    \frac{\ell_1\ell_2}{r^2} \pm \sqrt{\p[\Big]{1 - \frac{\ell_1^2}{r^2}}\p[\Big]{1 - \frac{\ell_2^2}{r^2}}}.
\end{equation*}
Thus, inequality \eqref{eq:cos-ineq} is equivalent to
\begin{equation*}
    \p[\Big]{\cos(\theta_1 - \theta_2) - \frac{\ell_1 \ell_2}{r^2}}^2 \leq \p[\Big]{1 - \frac{\ell_1^2}{r^2}}\p[\Big]{1 - \frac{\ell_2^2}{r^2}}.
\end{equation*}
If we write $\ell_1 = r \cos\alpha_1$, $\ell_2 = r \cos\alpha_2$, for $0 \leq \alpha_2 \leq \alpha_1 \leq \pi/2$, then this becomes
\begin{equation*}
    \abs[\big]{\cos(\theta_1 - \theta_2) - \cos\alpha_1\cos\alpha_2} \leq \sin\alpha_1 \sin\alpha_2
\end{equation*}
which is equivalent to
\begin{equation}
\label{eq:cos-ineq-2}
   \cos(\alpha_1 + \alpha_2) \leq \cos(\theta_1 - \theta_2) \leq \cos(\alpha_1 - \alpha_2).
\end{equation}
Recall that $\ell_1 \leq \ell_2$, so $\alpha_2 \leq \alpha_1$.
Note that we assumed that $\theta_1 - \theta_2 \in [0,\pi]$.
Therefore, the inequality \eqref{eq:cos-ineq-2} is equivalent to
\begin{equation*}
    \alpha_1 - \alpha_2 \leq \theta_1 - \theta_2 \leq \alpha_1 + \alpha_2.
\end{equation*}
The probability over $\theta_1$ and $\theta_2$ that $\ell \leq r$, given $\ell_1$ and $\ell_2$, is
\begin{align*}
    \frac{1}{4\pi^2} \int_{0}^{2\pi} \int_{0}^{2\pi} \ind_{\set{\alpha_1 - \alpha_2 \leq \theta_1 - \theta_2 \leq \alpha_1 + \alpha_2}} \d\theta_1 \d\theta_2
    = \frac{2\alpha_2}{2\pi} = \frac{\arccos(\ell_2/r)}{\pi}.
\end{align*}
Finally, we have
\begin{align*}
    \prob[\big]{\ell \leq r}
    &= 4 \prob[\big]{\ell \leq r \given \ell_1 \leq \ell_2 \leq r, \theta_1 - \theta_2 \in [0,\pi]} \\
    &= \frac{4}{\pi} \int_{0}^{r} \int_{0}^{\ell_2} \arccos(\ell_2/r) \d\mu(\ell_1) \d\mu(\ell_2) \\
    &= \frac{4}{\pi} \int_{0}^{r} \mu\p[\big]{[0,t]}\arccos(t/r) \d\mu(t),
\end{align*}
as claimed in \eqref{eq:atomless-transform}.
\end{proof}

Suppose that $F(x) = \mu\p[\big]{[0,x]}$ is the cumulative density function of $\mu$.
From \Cref{thm:analytic}, we get a recipe that transforms a probability measure on $\mu$ on $\RR_{\geq 0}$ into a new probability measure with cumulative density function $\cI(F)$ defined as
\begin{equation}
\label{eq:integral-transform-cdf}
    \cI(F)(x) \defined \frac{4}{\pi} \int_0^x F(t) F'(t) \arccos(t/x) \dt,
\end{equation}
assuming of course that $F'(x)$ exists.
Applying the Leibniz integral rule, we can obtain an integral transform that acts on probability density functions $\rho$ instead, given by
\begin{equation}
\label{eq:integral-transform-pdf}
    \cJ(\rho)(r)
    \defined \frac{4}{\pi r^2} \int_{0}^{r} \frac{ \p[\big]{\int_0^t \rho(y) \dy} \rho(t) t}{\sqrt{1 - t^2/r^2}} \dt.
\end{equation}

The same map can also be thought as sending a rotationally invariant probability measure $\mu$ on $\Graff_1(\RR^2)$ (which puts zero probability to lines through the origin) to a new one, defined by sampling two independent lines $L_1$ and $L_2$, looking at their intersection point $p$, and producing the line in $\Graff_1(\RR^2)$ through $p$ in such a way that the distance from said line to the origin is attained exactly at $p$.

The integral transforms defined by \eqref{eq:integral-transform-cdf} and \eqref{eq:integral-transform-pdf} appears to be new and gives a quick route to obtain \Cref{thm:karamata}.
Indeed, from \eqref{eq:cumulative}, we know that $\prob[\big]{\ell \leq r} = \cI(F)(r)$ where $F(x)$ is defined as
\begin{alignat*}{2}
    F(x) = \begin{cases}
        \frac{2 \arcsin x}{\pi} & \text{if $0 \leq x \leq 1$,} \\
        1 & \text{if $x \geq 1$.}
    \end{cases} & \; \implies \; &
    F'(x) = \begin{cases}
        \frac{2}{\pi \sqrt{1 - x^2}} & \text{if $0 \leq x \leq 1$,} \\
        0 & \text{if $x \geq 1$.}
    \end{cases}
\end{alignat*}
Thus, we have
\begin{align*}
    \prob[\big]{\ell \leq r} &= \frac{4}{\pi} \int_0^r F(t) F'(t) \arccos(t/r) \dt \\
    &= \frac{16}{\pi^3} \int_0^{\min\set{1,r}} \frac{\arccos(t/r)\arcsin(t)}{\sqrt{1 - t^2}} \dt.
\end{align*}

We remark that the condition on \Cref{thm:analytic} that the measure is atomless can be circumvented easily.
Indeed, we only used the atomless condition when we wrote that $\prob[\big]{\ell \leq r} = 2 \prob[\big]{\ell \leq r \given \ell_1 \leq \ell_2}$.
However, the following holds for general $\mu$:
\begin{equation*}
    \prob[\big]{\ell \leq r} = 2 \prob[\big]{\ell \leq r \given \ell_1 < \ell_2\leq r} + \prob[\big]{\ell \leq r \given \ell_1 = \ell_2\leq r}.
\end{equation*}

We conclude this session with the surprising result that says that $\ell$ is not integrable unless $\mu$ is trivial.

\begin{proof}[Proof of \Cref{cor:integrability}]
Recall in the proof of \Cref{thm:analytic} that we have computed the distance $\ell$ as a function of $\theta_1$,  $\theta_2$, $\ell_1$ and $\ell_2$, in equation \eqref{eq:distance}, as
\begin{equation*}
    \ell = \frac{\sqrt{\vphantom{\int}\ell_1^2 + \ell_2^2 - 2\ell_1 \ell_2 \cos(\theta_2-\theta_1)} }{\abs[\big]{\sin(\theta_2-\theta_1)}}.
\end{equation*}
By assumption, we have $\mu\p[\big]{(0,\infty)} > 0$, which implies that for some $a > 0$, we have that $\mu\p[\big]{[a,\infty)} > 0$.
Also note that, if $\ell_1, \ell_2 \geq a$, then
\begin{equation*}
    \sqrt{\ell_1^2 + \ell_2^2 - 2\ell_1 \ell_2 \cos(\theta_2-\theta_1)} \geq \ell_1 + \ell_2 \geq 2a.
\end{equation*}
Therefore, we have
\begin{align*}
    \Expec \ell &\geq \mu\p[\big]{[a,\infty)}^2 \cdot \expec[\big]{ \ell \given \ell_1,\ell_2 \geq a} \\
    &\geq \mu\p[\big]{[a,\infty)}^2 \cdot \expec[\Bigg]{ \frac{2a}{\abs[\big]{\sin(\theta_2-\theta_1)}} \given \ell_1,\ell_2 \geq a} \\
    &= \mu\p[\big]{[a,\infty)}^2 \cdot \frac{1}{2\pi} \int_{0}^{2\pi}\frac{2a}{\abs[\big]{\sin \gamma}} \d \gamma \\
    &\geq \mu\p[\big]{[a,\infty)}^2 \cdot \frac{a}{\pi} \int_{0}^{2\pi}\frac{1}{\gamma} \d \gamma,
\end{align*}
where we have used that $\abs{\sin x} \leq x$ for $x \geq 0$.
The result follows since we have $a > 0$, $\mu\p[\big]{[a,\infty)} > 0$ and $\int_0^{\eps} \dt / t = \infty$ for all $\eps > 0$.
\end{proof}


\section{Bertrand's paradox and beyond}
\label{sec:bertrand}

Recall that the usual discussion around Bertrand's  paradox suggest three different ways of sampling a random chord in a unit circle centered at the origin:
\begin{enumerate}
    \item[$I$] (Uniform Radius) A chord is selected such that its distance $\ell_{I}$ from the origin is uniform on the interval $[0,1]$.
    \item[$II$] (Uniform Midpoint) A chord is formed such that the midpoint of the chord is uniform on the unit disk.
    Write $\ell_{II}$ for the distance of this chord to the origin.
    \item[$III$] (Uniform Endpoints) A chord is formed by selecting the endpoints independently and uniformly in the unit circle. Write $\ell_{III}$ for the distance of this chord to the origin.
\end{enumerate}
The cumulative density functions of $\ell_{I}$, $\ell_{II}$ and $\ell_{III}$ are then, respectively,
\begin{align*}
    F_{I}(r) &= \prob[\big]{\ell_{I} \leq r} = r, \\
    F_{II}(r) &= \prob[\big]{\ell_{II} \leq r} = r^2, \\
    F_{III}(r) &= \prob[\big]{\ell_{III} \leq r} = \frac{2}{\pi} \arcsin r,
\end{align*}
when $0 \leq r \leq 1$.
Set $\mu_I$, $\mu_{II}$ and $\mu_{III}$ to be the probability measures on $\RR_{\geq 0}$ whose cumulative density functions are respectively $F_{I}$, $F_{II}$ and $F_{III}$.
Extending these chords to lines, these three measures correspond to rotationally symmetric probability measures $\cG(\mu_{I})$, $\cG(\mu_{II})$ and $\cG(\mu_{III})$ on $\Graff_1(\RR^2)$.

It is natural to ask what is the distribution of the intersection point of two independent random lines selected in each of these three ways.
If we choose the measure $\cG(\mu_{III})$, then we already know the answer from \Cref{thm:karamata}, as this is equivalent to the Karamata problem.

\begin{corollary}
\label{cor:bertrand}
For $j \in \set{I,II,II}$, write $\Prob_{j}\p[\big]{\ell \leq r}$ for the probability that $\ell \leq r$, where $\ell$ is the distance from the origin to the intersection point $L_1 \cap L_2$, when $L_1$ and $L_2$ are independent random lines sampled from $\cG(\mu_{j})$.
For $0 \leq r \leq 1$, we have
\begin{equation*}
    \Prob_{I}\p[\big]{\ell \leq r} = \frac{r^2}{2}, \; \Prob_{II}\p[\big]{\ell \leq r} = \frac{3r^4}{8}, \; \text{and } \; \Prob_{III}\p[\big]{\ell \leq r} = \frac{2}{\pi^2} \Li_2(r^2),
\end{equation*}
while for $r \geq 1$, we have
\begin{align*}
    \Prob_{I}\p[\big]{\ell \leq r} &= \frac{2\arccos(1/r)}{\pi} + \frac{r^2\arcsin(1/r)}{\pi} - \frac{\sqrt{r^2 - 1}}{\pi}, \\[0.5em]
    \Prob_{II}\p[\big]{\ell \leq r} &= \frac{2 \arccos(1/r)}{\pi} + \frac{3 r^4 \arcsin(1/r)}{4\pi} - \frac{(3r^2 + 2) \sqrt{r^2 - 1}}{4\pi}, \\[0.8em]
    \Prob_{III}\p[\big]{\ell \leq r} &= \frac{16}{\pi^3}\int_{0}^{1} \frac{ \arccos(t/r) \arcsin(t)}{\sqrt{1-t^2}} \dt.
\end{align*}
\end{corollary}
\begin{proof}
We will make use of the following integrals, which can be obtained via integration by parts.
For all $0 \leq x \leq r$, we have
\begin{alignat*}{2}
    \int_{0}^{x} t \arccos(t/r) \dt &= \frac{ x^2 \arccos(x/r)}{2} &&+ \frac{r^2\arcsin(x/r)}{4} - \frac{x \sqrt{r^2 - x^2}}{4} \\
    \int_{0}^{x} t^3 \arccos(t/r) \dt &= \frac{x^4 \arccos(x/r)}{4} &&+ \frac{3r^4\arcsin(x/r)}{32} \\
    &&&- \frac{x(3r^2 + 2x^2) \sqrt{r^2 - x^2}}{32}.
\end{alignat*}
As $\arccos(1) = \arcsin(0) = 0$ and $\arcsin(1) = \pi/2$, this gives
\begin{equation*}
    \int_{0}^{r} t \arccos(t/r) \dt = \frac{\pi r^2}{8}, \;\; \text{and} \;\;
    \int_{0}^{r} t^3 \arccos(t/r) \dt = \frac{3\pi r^4}{64}.
\end{equation*}

We now proceed with the proof.
From \Cref{thm:analytic} and \eqref{eq:integral-transform-cdf}, we have that
\begin{align*}
    \Prob_{I}\p[\big]{\ell \leq r} &= \cI(F_{I})(r) = \frac{4}{\pi} \int_0^x F_I(t) F'_I(t) \arccos(t/x) \dt \\
    &= \frac{4}{\pi} \int_0^{\min\set{1,r}} t \arccos(t/r) \dt.
\end{align*}
Therefore, for $0 \leq r \leq 1$, we have
\begin{align*}
    \Prob_{I}\p[\big]{\ell \leq r} = \frac{4}{\pi} \int_0^{r} t \arccos(t/r) \dt = \frac{4}{\pi} \cdot \frac{\pi r^2}{8} = \frac{r^2}{2},
\end{align*}
while for $r \geq 1$, we have
\begin{equation*}
    \Prob_{I}\p[\big]{\ell \leq r} = \frac{2\arccos(1/r)}{\pi} + \frac{r^2\arcsin(1/r)}{\pi} - \frac{\sqrt{r^2 - 1}}{\pi}.
\end{equation*}

Similarly, we have that
\begin{align*}
    \Prob_{II}\p[\big]{\ell \leq r} =\frac{8}{\pi} \int_0^{\min\set{1,r}} t^3 \arccos(t/r) \dt.
\end{align*}
Therefore, for $0 \leq r \leq 1$, we have
\begin{align*}
    \Prob_{II}\p[\big]{\ell \leq r} = \frac{8}{\pi} \int_0^{r} t^3 \arccos(t/r) \dt = \frac{8}{\pi} \cdot \frac{3 \pi r^4}{64} = \frac{3r^4}{8},
\end{align*}
while for $r \geq 1$, we have
\begin{equation*}
    \Prob_{II}\p[\big]{\ell \leq r} = \frac{2 \arccos(1/r)}{\pi} + \frac{3r^4\arcsin(1/r)}{4\pi} - \frac{(3r^2 + 2) \sqrt{r^2 - 1}}{4\pi}.
\end{equation*}

The value of $\Prob_{III}\p[\big]{\ell \leq r}$ follows from \Cref{thm:karamata}.
\end{proof}

So far we have only considered cases where the line must necessarily intersect the unit disk.
One natural distribution to consider is the one obtained by considering a point $(x,y)$ where both $x$ and $y$ are independent standard Gaussians.
The distance of $(x,y)$ to the origin then follows a Rayleigh distribution, whose cumulative density function and probability density function are given by:
\begin{alignat*}{2}
    F_{IV}(x) = 1 - e^{-x^2/2}  \; \text{and} \;  F'_{IV}(x) = x e^{-x^2/2}.
\end{alignat*}
Therefore, we have
\begin{align}
\nonumber
    \Prob_{IV}\p[\big]{\ell \leq r} &= \frac{4}{\pi} \int_0^r F_{IV}(t) F'_{IV}(t) \arccos(t/r) \dt \\
\label{eq:gaussian-integral}
    &= \frac{4}{\pi} \int_0^r \p[\big]{1 - e^{-t^2/2}} t e^{-t^2/2} \arccos(t/r) \dt.
\end{align}
See \Cref{fig:bertrand} for a comparison of the four probability measures found in this section.

\begin{figure}[!ht]
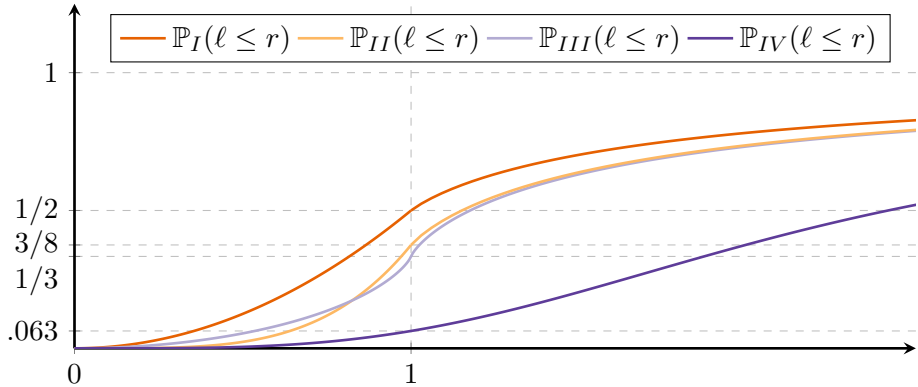

\centering

\caption{Different distributions obtained by changing the method of choosing a random line.}
\label{fig:bertrand}
\end{figure}

It turns out that the integral in \eqref{eq:gaussian-integral} can also be evaluated in terms of known functions, in this case, of the (modified) Bessel function of the first kind $I_0(x) \from \RR \to \RR$ defined by the convergent series
\begin{equation*}
    I_0(x) \defined \sum_{k=0}^{\infty} \frac{(x^2/4)^k}{(k!)^2}.
\end{equation*}
The reader will be spared from the long computation, which can be done by expanding $\arccos$ in series integrating term by term\footnote{Or just type \texttt{(4/Pi)Integrate[(1 - E\^{}(-t\^{}2/2))tE\^{}(-t\^{}2/2) ArcCos[t/r], \{t, 0, r\}]} in Mathematica or another computer algebra system.}, which with some effort leads to
\begin{equation*}
    \Prob_{IV}\p[\big]{\ell \leq r} = 1 - 2 e^{-r^2/4} I_0\p[\big]{r^2/4} + e^{-r^2/2} I_0\p[\big]{r^2/2}.
\end{equation*}

\Cref{fig:bertrand} highlights how different notions of \emph{random chord} leads to different distributions for the intersection point.
The first three methods only produce lines that intersect the unit circle, so it is not surprising that their intersection point is more likely to attain smaller values than the Gaussian case.
Nevertheless, \Cref{cor:integrability} implies that $\Expec \ell = \infty$ in all these four cases.


\section{An analytic identity}
\label{sec:identity}

In this section we prove that the expression \eqref{eq:karamata-extended} in \Cref{thm:karamata} indeed simplifies to the dilogarithm whenever $0 \leq r \leq 1$.
In other words, we show that for $0 \leq r \leq 1$,
\begin{align*}
    \prob[\big]{\ell \leq r}
    &= \frac{16}{\pi^3}\int_0^{\min\set{1,r}} \frac{ \arccos(t/r) \arcsin(t)}{\sqrt{1-t^2}} \dt = \frac{2}{\pi^2} \Li_2(r^2).
\end{align*}
To simplify the integral above, we integrate by parts taking $v(x) = (\arcsin x)^2/2$ and $u(x) = \arccos(x/r)$, so $u'(x) = -1 / \sqrt{r^2 - t^2}$ and $v'(x) = \arcsin x / \sqrt{1 - x^2}$.
Since $v(0) = 0$ and $u(r) = 0$, we have
\begin{align*}
    \int_{0}^{r} \frac{ \arccos(t/r) \arcsin(t)}{\sqrt{1-t^2}} \dt
    &= \frac{1}{2}\int_0^r \frac{(\arcsin t)^2}{\sqrt{r^2 - t^2}}\dt
    = \frac{1}{2} \int_0^1 \frac{\p[\big]{\arcsin(rt)}^2}{\sqrt{1-t^2}}\dt.
\end{align*}

To finalize the computation, we recall two standard facts.
The first is the Taylor series for the $(\arcsin x)^2$ function, valid for $\abs{x} \leq 1$, which is given by
\begin{align*}
    (\arcsin x)^2 = \frac{1}{2} \sum_{n=1}^{\infty} \frac{(2x)^{2n}}{n^2 \binom{2n}{n}}.
\end{align*}
The second are the even values of the Wallis integral, namely,
\begin{align*}
    \int_{0}^{1} \frac{x^{2n}}{\sqrt{1-x^2}}\dx = \int_{0}^{\pi/2} (\cos x)^{2n} \dx  = \frac{\pi}{2} \cdot \frac{\binom{2n}{n}}{4^n}.
\end{align*}
Together, we obtain that
\begin{align*}
    \int_{0}^{1} \frac{\p[\big]{\arcsin(rt)}^2}{\sqrt{1-t^2}} \dt
    &= \frac{1}{2} \int_{0}^{1} \sum_{n=1}^{\infty} \frac{(2rt)^{2n}}{n^2 \binom{2n}{n} \sqrt{1-t^2}} \du \\
&= \frac{1}{2} \sum_{n=1}^{\infty} \frac{(2r)^{2n}}{n^2 \binom{2n}{n}} \int_{0}^{1} \frac{t^{2n}}{\sqrt{1-t^2}} \dt \\
&= \frac{\pi}{4} \sum_{n=1}^{\infty} \frac{r^{2n} }{n^2} = \frac{\pi}{4} \Li_2(r^2).
\end{align*}

Putting all together, we have, for $0 \leq r \leq 1$, that
\begin{align*}
    \prob[\big]{\ell \leq r}
    &= \frac{16}{\pi^3}\int_{0}^{r} \frac{ \arccos(t/r) \arcsin(t)}{\sqrt{1-t^2}} \dt \\
    &= \frac{16}{\pi^3} \cdot \frac{1}{2}\int_{0}^{1} \frac{\p[\big]{\arcsin(rt)}^2}{\sqrt{1-t^2}} \dt \\
    &= \frac{8}{\pi^3} \cdot \frac{\pi}{4} \Li_2(r^2) = \frac{2}{\pi^2} \Li_2(r^2).
\end{align*}


\section{Further directions}
We discuss a few avenues that remain unexplored.

\subsection{Hyperbolic volumes}

The dilogarithm function is closely related to the volume of ideal tetrahedra in hyperbolic space (see Section 4 of Zagier~\cite{Zagier2007-tc}).
This suggest the possibility of obtaining \Cref{thm:karamata}, say when $0 \leq r \leq 1$, directly recovered geometrically from this interpretation.

\subsection{Special values}

We can compute $\prob[\big]{\aff(X_1,X_2)\cap\aff(X_3,X_4) \in S}$ explicitly for some shapes $S \subseteq \RR^2$ from equation \eqref{eq:explicit}.
Indeed, we can do so for disks centered at the origin, cones with a vertex at the origin, as well as boolean combinations of these shapes.
Are there any other simple shapes, such as squares, triangles or ellipses, such that this probability can also be expressed in terms of known functions?
Can our new probabilistic interpretation of the dilogarithm lead to new explicit values of $\Li_2(x)$?

\subsection{Polylogarithms} Is there a related probabilistic interpretation to the polylogarithm functions $\Li_m(z)$, defined as $\sum_{n = 1}^{\infty} z^n / n^m$ for $\abs{z} < 1$?

\subsection{Higher dimensions}

Write $\Graff(\RR^d)$ for the set of all affine subspaces of $\RR^d$ and $\Graff_k(\RR^d)$ for those subspaces of dimension $k$.
For $X,Y \in \Graff(\RR^d)$, write $X \meet Y$ for their intersection and $X \join Y$ for the affine span of $X$ and $Y$.
This this turns $\Graff(\RR^d)$ into a lattice.
While there are some natural invariant measures on $\Graff(\RR^d)$, see for instance Klain and Rota~\cite{Klain1997-ux}, they are not probability measures since $\Graff(\RR^d)$ is not a compact space.

What we have explored in our work boils down to understanding the distribution of $(\cX_{1,1} \join \cX_{1,2}) \meet (\cX_{2,1} \join \cX_{2,2})$, where $\cX_{1,1}$, $\cX_{1,2}$, $\cX_{2,1}$ and $\cX_{2,2}$ are independent random elements of $\RR^2 \cong \Graff_0(\RR^2)$ with same distribution $\eta$.
In the case of \Cref{thm:karamata}, $\eta$ is the uniform distribution on $\SS^1 \subseteq \RR^2$, while in \Cref{sec:bertrand}, we also considered the case where $\eta$ is the Gaussian distribution on $\RR^2$.

We propose a three dimensional version of \Cref{thm:karamata} as follows.
Let $\cX_{1,1}$, $\cX_{1,2}$, $\cX_{1,3}$, $\cX_{2,1}$, $\cX_{2,2}$, $\cX_{2,3}$, $\cX_{3,1}$, $\cX_{3,2}$ and $\cX_{3,3}$ be independently and uniformly distributed in $\SS^2 \subseteq \RR^3 \cong \Graff_0(\RR^3)$ and compute the distribution of
\begin{equation*}
    \p[\big]{\cX_{1,1} \join \cX_{1,2} \join \cX_{1,3}}
    \meet \p[\big]{\cX_{2,1} \join \cX_{2,2} \join \cX_{2,3}}
    \meet \p[\big]{\cX_{3,1} \join \cX_{3,2} \join \cX_{3,3}}.
\end{equation*}
The natural $d$-dimensional generalization then would be about
\begin{equation}
\label{eq:higher-dimensions}
    \p[\big]{\cX_{1,1} \join \dotsb \join \cX_{1,d}} \meet \dotsb \meet \p[\big]{\cX_{d,1} \join \dotsb \join \cX_{d,d}},
\end{equation}
where $\cX_{i,j}$, $1 \leq i,j \leq d$, are taken independently and uniformly form $\SS^{d-1} \subseteq \RR^d$.
It is also interesting to consider other probability measures for $\cX_{i,j}$, such as the $d$-dimensional Gaussian measure.

More abstractly, if we are given a probability measure on $\Graff(\RR^d)$ and any expression such as $(X_1 \meet X_2) \join (X_3 \meet X_4 \meet X_5) \join X_6$, we can build a new measure on $\Graff(\RR^d)$ by substituting each variable $X_i$ by a independent random element $\cX_i$ sampled from our initial measure on $\Graff(\RR^d)$.
This leads to many variations that could be of interest.
For instance, understanding the distance from the origin to the $(d-2)$-dimensional affine subspace
\begin{equation*}
    \p[\big]{\cX_{1,1} \join \dotsb \join \cX_{1,d}}  \meet \p[\big]{\cX_{2,1} \join \dotsb \join \cX_{2,d}}
\end{equation*}
in $\RR^d$ could be an important stepping stone towards understanding \eqref{eq:higher-dimensions}.

\subsection{Extremal questions}
In \Cref{thm:karamata}, we compute the distribution of the distance $\ell$ of $\aff(X_1, X_1') \cap \aff(X_2, X_2')$ to the origin, where $X_1$, $X_1'$, $X_2$ and $X_2'$ are independent random points chosen uniformly on the circle $\SS^1 \subseteq \RR^2$.
What if they are all chosen according to another probability measure $\nu$ in $\SS^1$ instead?
Given a fixed $r \geq 0$, how large can $\prob[\big]{\ell \leq r}$ be if we can choose $\nu$ as we like?
Which measures $\nu$, if any, attain such bound?
This time, \Cref{thm:analytic} is not applicable since we no longer have rotational invariance.
These problems were considered by the first author and Ramos in~\cite{Bortolotto2024-oq}, where for instance, they show that any non-atomic $\nu$ is an extremizer for $r = 1$.


\section*{Acknowledgements}

The authors are grateful to Emmanuel Kowalski for showing us the original manuscript of Karamata and thus sparking our interest in the subject.
We are also grateful for the many fruitful discussions with João Pedro Ramos.

The second author is partially supported by ERC Starting Grant 101163189 and UKRI Future Leaders Fellowship MR/X023583/1.


\emergencystretch=1em

\printbibliography


\end{document}